\documentclass[reqno]{amsart}
\usepackage[english]{babel}
\usepackage{graphicx,subfigure}
\usepackage{amsmath,amssymb}
\usepackage{appendix}
\usepackage{color}
\definecolor{oneblue}{rgb}{0,0.0,0.75}

\usepackage{fancyhdr}

\makeatletter
\newcommand{\sech}{\mathop{\operator@font sech}}
\newcommand{\sign}{\mathop{\operator@font sign}}
\makeatother

\newtheorem{lemma}{Lemma}[section]
\newtheorem{theorem}{Theorem}[section]
\newtheorem{proposition}{Proposition}[section]

\newtheorem{remark}{Remark}[section]
\numberwithin{equation}{section}

\begin{document}

\title[]{Spectral Jacobi approximations for Boussinesq systems}

\author[A. Duran]{Angel Duran}
\address{ Applied Mathematics Department,  University of
Valladolid, 47011 Valladolid, Spain}
\email{angeldm@uva.es}

\subjclass[2010]{65M70 (primary), 76B15 (secondary)}
\keywords{Boussinesq systems, spectral methods, error estimates}


\begin{abstract}
This paper is concerned with the numerical approximation of initial-boundary-value problems of a three-parameter family of Bona-Smith systems, derived as a model for the propagation of surface waves under a physical Boussinesq regime. The work proposed here  is focused on the corresponding problem with Dirichlet boundary conditions and its approximation in space with spectral methods based on Jacobi polynomials, which are defined from the orthogonality with respect to some weighted $L^{2}$ inner product.
Well-posedness of the problem on the corresponding weighted Sobolev spaces is first analyzed and existence and uniqueness of solution, locally in time, are proved. Then the spectral Galerkin semidiscrete scheme and some detailed comments on its implementation are introduced. The existence of numerical solution and error estimates on those weighted Sobolev spaces are established. Finally, the choice of the time integrator to complete the full discretization takes care of different stability issues that may be relevant when approximating the semidiscrete system. Some numerical experiments illustrate the results.

\end{abstract}

\maketitle

\section{Introduction}\label{sec1}
\subsection{Boussinesq systems for surface wave propagation}\label{sec11}
The four parameter family of Boussinesq systems of water waves
\begin{eqnarray}
\eta_{t}+u_{x}+(\eta u)_{x}+au_{xxx}-b\eta_{xxt}&=&0,\nonumber\\
u_{t}+\eta_{x}+uu_{x}+c\eta_{xxx}-du_{xxt}&=&0,\label{BS1}
\end{eqnarray}
was introduced in \cite{BonaChS2002,BonaChS2004} as a 1D approximation of the Euler equations for the two-way propagation of an irrotational, free surface flow of an incompressible inviscid fluid in a uniform horizontal channel, \cite{Lannes}, and the approximation is valid for long, small-amplitude waves, in the sense that
$$\epsilon=\frac{a}{h}<<1,\frac{\lambda}{h}>>1,\; S=\frac{a\lambda^{2}}{h^{2}}\sim 1,$$ where $a$ and $\lambda$ denote, respectively, the amplitude and wavelength of the surface wave and $h$ is the depth of the channel. In (\ref{BS1}), $x$ and $t$ are proportional to the position along the channel and time, respectively, while the real-valued functions $\eta=\eta(x,t), u=u(x,t)$ represent, respectively, the free surface elevation at $(x,t)$ with respect to an equlibrium and the horizontal velocity of the fluid at a height $y=-1+\theta(1+\eta(x,t))$, for some $\theta\in [0,1]$ (with the bottom of the channel at $\theta=0$ and the free surface at $\theta=1$. In dimensionless variables, $h=-1$ is fixed.) The parameters $a, b, c, d$ are given by
\begin{eqnarray}
a=\frac{1}{2}(\theta^{2}-\frac{1}{3})\nu,\; b=\frac{1}{2}(\theta^{2}-\frac{1}{3})(1-\nu),\nonumber\\
c=\frac{1}{2}(1-\theta^{2})\mu,\; d=\frac{1}{2}(1-\theta^{2})(1-\mu),\label{BS1p}
\end{eqnarray}
where $\nu$ and $\mu$ are real constants. See \cite{DougalisM2008} for a review on the systems (\ref{BS1}). 

The present paper is focused on the so-called Bona-Smith family, \cite{BonaS76}, of the form (\ref{BS1}) with $\nu=0$ and $b=d$. This leads to
\begin{eqnarray}
a=0, b=d=\frac{3\theta^{2}-1}{6}, c=\frac{2-3\theta^{2}}{3},\label{BS1p2}
\end{eqnarray} 
with $2/3\leq \theta^{2}\leq 1$, because the system is linearly ill-posed if $\theta^{2}<2/3$.  (Note also that $b=d>0, c\leq 0$.) Limiting cases are $\theta^{2}=2/3$ ($a=c=0, b=d=1/6$ or BBM-BBM system) and $\theta^{1}=1$ ($a=0, b=d=1/3, a=-1/3$). See \cite{BonaS76,BonaChS2004} for well-posedness results of the corresponding initial-value problem (ivp).
One topic of research on (\ref{BS1}) concerns the dynamics of its solitary wave solutions, \cite{DougalisDLM2007}. The study is usually made via the numerical approximation to the corresponding periodic ivp of (\ref{BS1}). However, some other phenomena which can be modelled with (\ref{BS1}) consider the system on a bounded interval $(a,b)$ and with other types of boundary conditions for $u$ and $\eta$ at $x=a, b$ for $t>0$. More specifically, here we will focus on Dirichlet boundary conditions for $\eta$ and $u$.

The present work is greatly indebted to mainly two papers by Antonopoulos et al., \cite{ADM1,ADM2} and one by Bj{\o}rkavag et al., \cite{Kalisch}. In \cite{ADM1}, the authors analyze the well-posedness of several ibvp for the Bona-Smith systems on bounded spatial intervals $[-L,L], L>0$. For the case of Dirichlet boundary conditions, the main result is the existence and uniqueness of solution, locally in time, on suitable Sobolev spaces and on spaces of  continuously differentiable functions of certain order. Before that, Bona and Chen, \cite{BonaChen}, proved a result of local well-posedness for the initial-boundary-value problem ( ibvp) with Dirichlet boundary conditions of the BBM-BBM system. On the other hand, \cite{ADM2} is concerned with the numerical approximation to the ibvp for the Bona-Smith systems analyzed in \cite{ADM1}.The numerical approximation is based on the standard Galerkin-finite element method with piecewise polynomial functions for the spatial discretization and a fourth-order explicit Runge-Kutta scheme as time integrator. Error estimates for the semidiscrete approximation and convergence of the full discretization are proved. The accuracy of the scheme is used to investigate computationally the interaction between solitary wave solutions and the solitary waves with the boundaries. Other error estimates, in the case of Dirichlet boundary conditions, are obtained in \cite{Winther} for the  case $\theta^{2}=1$ and the semidiscrete approximation with a nonstandard Galerkin finite element method; in \cite{BonaChen} for the BBM-BBM system and a fully discrete finite difference method of fourth order in space and time; and in \cite{Chatzi} for the BBM-BBM system when discretizing with a the Galerkin finite element method in space and a explicit multistep method in time.

The present paper proposes to approximate the Bona-Smith systems (\ref{BS1}), (\ref{BS1p2}) with Dirichlet boundary conditions by using spectral methods based on Jacobi polynomials for the spatial discretization and time discretization with special stability properties. The semidiscretization in space is motivated by 
\cite{Kalisch}, where a particular family of Jacobi polynomials, the Legendre polynomials, was already considered for approximating the ibvp of the BBM-BBM system, in order to study the onset of wave breaking connected with undular bores. The spectral approach based on Jacobi polynomials involves the use of suitable weighted inner products, \cite{BernardiM1997}. This requires some previous analysis of the well-posedness of the problem on the corresponding weighted Sobolev spaces. On the other hand, the Bona-Smith system contains regularizing terms of BBM type, $b\eta_{xxt}$ and $b u_{xxt}$. As a consequence, the spectral semidiscrete system is not stiff when $c=0$ and may have some stiff component when $c<0$, due to the presence of the term $c\eta_{xxx}$. This should have influence in the error estimate of the semidiscretization and in the choice of the time integrator. In addition, the last point may also be determined by the control of some stability issues of relevance, specially for long time simulations and when approximating problems from nonregular data. In the present paper we focused on the order of dispersion, \cite{SomH}, and the so-called strong stability preserving (SSP) property, \cite{Gotlieb2005,GotliebKS2009}, related to the linear and nonlinear stability properties that prevent spurious oscillatory behaviour in the temporal discretization of the semidiscrete system.
\subsection{Highlights and structure}\label{sec12}
The main goals of the project are:
\begin{enumerate}
\item The study of well-posedness of the corresponding ibvp with the proof of a result of existence, uniqueness, and regularity of solution, locally in time, on suitable weighted Sobolev spaces. The proof is based on similar arguments to those considered in \cite{ADM1} (in particular, the integral formulation of the problem) but adapted, in a nontrivial way, to weighted Sobolev spaces. This is treated in section \ref{sec2}. 
\item Section \ref{sec3} develops a convergence analysis of the semidiscrete Galerkin method based on spectral Jacobi polynomials, with the derivation of error estimates depending on the Jacobi weight and the regularity of the data. The lines of the proof involve a weak formulation of a related problem with homogeneous boundary conditions. In addition, several details on the implementation are provided. More specifically, the Galerkin-Numerical Integration (G-NI) formulation is described, and the relation with the spectral collocation approach is explained, \cite{CanutoHQZ1988}.
\item In Section \ref{sec4} the time integrators are introduced and their stability properties are described. The performance of the resulting full discretizations is checked, illustrating in particular 
the error estimates obtained in the analysis of the semidicretization in terms of the regularity of the solution. 
\end{enumerate}

\subsection{On spectral methods in weighted Sobolev spaces}\label{sec13}
We end the introduction summarizing some properties on weighted Sobolev spaces and Jacobi polynomials that will be used throughout the paper.  We refer to, e.~g., \cite{GottliebO,Mercier,MadayQ1981,BernardiM1989,CanutoHQZ1988,CanutoQ1982a,BernardiM1997,ShenTW2011} for details and additional properties.

Let $-1<\mu<1$ and define the Jacobi weight function 
\begin{eqnarray}
w(x)=w_{\mu}(x)=(1-x^{2})^{\mu},\; x\in\Omega=(-1,1).\label{jacobiw}
\end{eqnarray}
Then 
$L_{w}^{2}=L_{w}^{2}(\Omega)$ will stand for the space of squared integrable functions with respect to the weighted inner product
\begin{eqnarray}
(\phi,\psi)_{w}=\int_{-1}^{1}\phi(x)\psi(x)w(x)dx,\; \phi,\psi\in L_{w}^{2},\label{12a}
\end{eqnarray}
and associated norm $||\phi||_{0,w}=(\phi,\phi)_{w}^{1/2}$. For the Sobolev spaces $H_{w}^{k}=H_{w}^{k}(\Omega), k\geq 0$ integer (where $H_{w}^{0}=L_{w}^{2}$) the corresponding norm will be denoted by
\begin{eqnarray*}
||\phi ||_{k,w}^{2}=\sum_{j=0}^{k}||\frac{d^{j}}{dx^{j}}\phi||_{0,j}^{2}.
\end{eqnarray*}
We will also consider the spaces $H_{w,0}^{k}=H_{w,0}^{k}(\Omega)$ of functions $\phi\in H_{w}^{k}$ such that $\phi(-1)=\phi(1)=0$. For $s\geq 0$, $H_{w}^{s}=H_{w}^{s}(\Omega)$ (and $H_{w,0}^{s}=H_{w,0}^{s}(\Omega)$) are defined from interpolation theory, \cite{Adams,BernardiM1997}. Note that when $w(x)=1$ the spaces $H_{w}^{s}, H_{w,0}^{s}$ are the standard Sobolev spaces $H^{s}, H_{0}^{s}$.

For an integer $N\geq 2$, $\mathbb{P}_{N}$ will stand for the space of polynomials of degree at most $N$ on $\overline{\Omega}=[-1,1]$ and
\begin{eqnarray*}
\mathbb{P}_{N}^{0}=\{p\in\mathbb{P}_{N} / p(-1)=p(1)=0\}.
\end{eqnarray*}

Associated to (\ref{jacobiw}) we define the family of Jacobi polynomials $\{J_{n}^{\mu}\}_{n=0}^{\infty}$, which are orthogonal to each other in $L_{w}^{2}$. Particular cases such as Legendre and Chebyshev families correspond to $\mu=0$ and $\mu=-1/2$, respectively. Most properties of this Jacobi family (a particular one of the more general Jacobi polynomials $\{J_{n}^{\mu,\nu}\}_{n=0}^{\infty}$, orthogonal in $L_{w_{\mu,\nu}}^{2}$ with $w_{\alpha,\beta}(x)=(1-x)^{\alpha}(1+x)^{\beta}, -1<\alpha,\beta<1$) are extensions of the corresponding properties of the Legendre family, \cite{BernardiM1989,BernardiM1997}.

Let $N\geq 2$ be an integer and consider 
the bilinear form 
\begin{eqnarray}
A_{b}(\phi,\psi)=(\phi,\psi)_{w}+L_{b}(\phi,\psi),\label{ad26},\; \phi, \psi\in H_{w,0}^{1},
\end{eqnarray}
where, for $b>0$
\begin{eqnarray}
L_{b}(\phi,\psi)=\int_{-1}^{1}b\partial_{x}\phi\cdot\partial_{x}(\psi w)d{x}.\label{ad24}
\end{eqnarray}
Note that $L_{b}$ is equivalent to
\begin{eqnarray}
L(\phi,\psi)=\int_{\Omega}\partial_{x}\phi\cdot\partial_{x}(\psi w)d{x},\label{ad26a}
\end{eqnarray}
and therefore, \cite{BernardiM1989,BernardiM1997,CanutoHQZ1988}, 
the bilinear form $A_{b}$ in (\ref{ad26}) is continuous 
on $H_{w}^{1}\times H_{w,0}^{1}$ and elliptic on 
$H_{w,0}^{1}\times H_{w,0}^{1}$, that is, there are 
positive constants $C_{1}, C_{2}$ such that for all 
$\phi, \psi\in H_{w,0}^{1}$
\begin{eqnarray*}
|A_{b}(\phi,\psi)|&\leq &C_{1} \left(||\phi||_{0,w}||\psi||_{0,w}+||\phi_{x}||_{0,w}||\psi_{x}||_{0,w}\right)\\
&\leq &C_{1} ||\phi||_{1,w}||\psi||_{1,w},\; \phi\in H_{w}^{1}, \psi\in H_{w,0}^{1},\\
A_{b}(\phi,\phi)&\geq &C_{2} ||\phi||_{1,w}^{2},\;\phi\in H_{w,0}^{1}.
\end{eqnarray*}
If 
$v\in H_{w,0}^{1}$, then the orthogonal projection 
 of $v$ with 
respect to $A$ is defined as $R_{N}v\in 
\mathbb{P}_{N}^{0}$ such that
\begin{eqnarray}
A_{b}(R_{N}v-v,\psi)=0,\; \psi\in \mathbb{P}_{N}^{0}.\label{ad27}
\end{eqnarray}
For this projection, we have, \cite{BernardiM1989}
\begin{eqnarray}
||v-R_{N}v||_{1,w}+N||v-R_{N}v||_{0,w}\leq C N^{1-m}||v||_{m,w},\label{ad27b}
\end{eqnarray}
for $v\in H_{w}^{m}\cap H_{w,0}^{1}, m\geq 1,$ and some constant $C$. Furthermore, \cite{EAAD20}, the 
generalized estimate
\begin{eqnarray}
||v-R_{N}v||_{2,w} \leq C N^{3-m}||v||_{m,w},\label{ad27d}
\end{eqnarray}
holds, for some constant $C$ and $m\geq 3$.

For an integer $m\geq 0$, $C^{m}(\Omega)$ will stand for the space of $m$th-order continuously differentiable functions on $\Omega$. In addition, for $T>0, m\geq 0$, $C_{T}^{m}=C(0,T,H_{w}^{m})$ will denote the space of continuous functions $f:[0,T]\rightarrow H_{w}^{m}$ with norm
$$||f||_{C_{T}^{m}}=\max_{0\leq t\leq T}||f(t)||_{m,w},$$ while for product spaces of the form $E=C_{T}^{m}\times C_{T}^{p}, m,p\geq 0$, we will consider the norm
\begin{eqnarray}
||(f,g)||_{E}=||f||_{C_{T}^{m}}+||g||_{C_{T}^{p}}.\label{Enorm}
\end{eqnarray}
Throughout the paper $C$ will be a generic constant, whose dependence on some parameters will be specified in each case if necessary.

\section{Well-posedness}
\label{sec2}
\subsection{The initial-boundary-value problem}
We will consider the following ibvp for the Bona-Smith system on $\Omega$
\begin{eqnarray}
\eta_{t}+u_{x}+(\eta u)_{x}-b\eta_{xxt}&=&0,\nonumber\\
u_{t}+\eta_{x}+uu_{x}+c\eta_{xxx}-bu_{xxt}&=&0,\; x\in\overline{\Omega}, t>0,\label{BS2}\\
\eta(x,0)=\eta_{0}(x), u(x,0)=u_{0}(x),&&x\in\overline{\Omega},\nonumber\\
\eta(-1,t)=h_{1}(t),\; u(-1,t)=v_{1}(t),&&t>0,\label{BS2b}\\
\eta(1,t)=h_{2}(t),\; u(1,t)=v_{2}(t),&&t>0,\label{BS2c}
\end{eqnarray}
for continuous functions $h_{i}, v_{i}, i=1,2$.
\begin{remark}
\label{remark20}
The ibvp (\ref{BS2})-(\ref{BS2c}) can be posed on any interval $I=(a,b)$ using the linear change of variables between $I$ and $\Omega$
$$y=Ax+B,\quad x\in \Omega,\quad A=\frac{b-a}{2}, B=\frac{b+a}{2}.$$ For $-1<\mu<1$, the corresponding weight in $I$ would be $w_{I}(y)=w\left((y-B)/A\right)=\left(1-\left(\frac{y-B}{A}\right)^{2}\right)^{\mu}$.
\end{remark}
In \cite{ADM1} the following well-posedness result for the classical solutions of (\ref{BS2}) is proved.
\begin{theorem}
\label{theoADM}
Let $0<T<\infty$, $\eta_{0}\in C^{3}, u_{0}\in C^{2}, h_{i}, v_{i}\in C^{1}([0,T]), i=1,2$, and supposed that the compatibility conditions
$$\eta_{0}(-L)=h_{1}(0), \eta_{0}(L)=h_{2}(0), u_{0}(-L)=v_{1}(0), u_{0}(L)=v_{2}(0),$$ are satisfied. Let
$$\beta_{0}:=||\eta_{0}||_{C^{1}}+||u_{0}||_{C^{0}}+\max_{0\leq t\leq T}\sum_{i=1}^{2}(|h_{i}(t)|+|v_{i}(t)|).$$ Then, there exists a $T_{0}=T_{0}(T,\beta_{0})\in (0,T]$ such that (\ref{BS2}) has a unique classical solution $(\eta, u)$ in $[0,T_{0}]$ with $\eta\in C(0,T_{0};C^{3}), u\in C(0,T_{0};C^{2})$
\end{theorem}
cf. also Remark 2.2 of the same reference. For the BBM-BBM case, see also \cite{BonaChen}.

We extend Theorem \ref{theoADM} by proving the corresponding result on weighted Sobolev spaces with weights of the form (\ref{jacobiw}) and in $\Omega$. 
The proofs are based on the arguments used in \cite{ADM1} and two additional lemmas. The first one is given in \cite{BernardiM1989}.
\begin{lemma}
\label{lemma21} Let $-1<\mu<1$ and $w(x)$ be given by (\ref{jacobiw}). If $f\in H_{w}^{1}(\Omega)$, then $f\in H^{r}(\Omega)$ where $r=\inf\{1,1-\mu/2\}$. Consequently, $f\in C^{0}(\overline{\Omega})$.
\end{lemma} 
\begin{remark}
Observe that, since $-1<\mu<1$ then $r>1/2$ and therefore $H^{r}$ is an algebra.
\end{remark}
\begin{lemma}
\label{lemma22} If $\eta, u\in H_{w}^{1}$ then $\eta u\in H_{w}^{1}$ and
\begin{eqnarray*}
||\eta u||_{1,w}\leq C||\eta ||_{1,w}||u||_{1,w},\label{algebraw}
\end{eqnarray*}
for some constant $C>0$
\end{lemma} 
\begin{proof}
Using Sobolev inequality and Lemma \ref{lemma21}, it holds that
\begin{eqnarray*}
||\eta u||_{1,w}^{2}&=&\int_{-1}^{1}(\eta u)^{2}wdx+\int_{-1}^{1}\left(\eta'u+\eta u'\right)^{2}wdx\\
&\leq &\int_{-1}^{1}(\eta u)^{2}wdx+2\int_{-1}^{1}\left((\eta')^{2}u^{2}+\eta^{2}( u')^{2}\right)wdx\\
&\leq &2\int_{-1}^{1}\left(\eta^{2}+(\eta')^{2}\right)u^{2}wdx+2\int_{-1}^{1}\eta^{2}(u')^{2}wdx\\
&\leq &2|u|_{\infty}^{2}\int_{-1}^{1}\left(\eta^{2}+(\eta')^{2}\right)wdx+2|\eta|_{\infty}^{2}\int_{-1}^{1}\left(u^{2}+(u')^{2}\right)wdx\\ 
&\leq & C\left(||u||_{r}^{2}||\eta||_{1,w}^{2}+||\eta||_{r}^{2}||u||_{1,w}^{2}\right)\leq C ||\eta||_{1,w}^{2}||u||_{1,w}^{2},
\end{eqnarray*}
where $r=\inf\{1,1-\mu/2\}$. 
\end{proof}

We now consider the two-point boundary-value problem, for $b>0$
\begin{eqnarray}
&&v-bv''=-f',\quad x\in \overline{\Omega},\nonumber\\
&&v(-1)=v(1)=0,
\label{jacob23}
\end{eqnarray}
whose solution can be written in the form, cf. \cite{ADM1}
\begin{eqnarray}
v(x)=(A_{D}f)(x)=\int_{-1}^{1}G_{\xi}(x,\xi)f(\xi)d\xi,\label{jacob23b}
\end{eqnarray}
where $G:\overline{\Omega}\times \overline{\Omega}\rightarrow \mathbb{R}$ is the Green function
\begin{eqnarray*}
G(x,\xi)=-\frac{1}{bW}\left\{
\begin{matrix}w_{1}(\xi)w_{2}(x)&-1\leq\xi\leq x\\w_{1}(x)w_{2}(\xi)&x<\xi\leq 1\end{matrix}
\right.
\end{eqnarray*}
where $w_{1}(x)={\rm sinh}\frac{1+x}{\sqrt{b}}, w_{2}(x)={\rm sinh}\frac{1-x}{\sqrt{b}}$, and $W=w_{1}w_{2}'-w_{1}'w_{2}$. Let $w$ be given by (\ref{jacobiw}) for some $-1<\mu<1$ and consider the weak formulation of (\ref{jacob23})
\begin{eqnarray}
A_{b}(v,\varphi)=-\left(f',\varphi\right)_{0,w},\quad \varphi\in H_{w,0}^{1},\label{jacob24}
\end{eqnarray}
where $A_{b}$ is defined in (\ref{ad26}), (\ref{ad24}). (Note that $w_{1}, w_{2}\in H_{w}^{2}$.) Due to the elliptic properties of $A_{b}$, mentioned above, Lax-Milgram theorem ensures the existence of a solution $v\in H_{w,0}^{1}\cap H_{w}^{2}$ of (\ref{jacob24}) such that
\begin{eqnarray*}
||v||_{2,w}\leq C||f||_{0,w}\leq C||f||_{1,w},\quad f\in H_{w}^{1}.
\end{eqnarray*}
Furthermore, taking $\varphi=v$ in (\ref{jacob24}) and using again the ellipticity of $A_{b}$, it hold that
\begin{eqnarray*}
||v||_{1,w}^{2}\leq C|(f',v)_{0,w}|\leq C||f||_{0,w}||v||_{1,w},
\end{eqnarray*}
and therefore $||v||_{1,w}\leq C||f||_{0,w}$, for some constant $C$. This proves the following result (cf. Lemma 2.1 of \cite{ADM1}).
\begin{lemma}
\label{lemma23} Let $A_{D}$ be defined by (\ref{jacob23b}). Then:
\begin{itemize}
\item[(i)] If $f\in L_{w}^{2}$, then $A_{D}f\in H_{w,0}^{1}$ and $||A_{D}f||_{1,w}\leq C ||f||_{0,w}$ for some constant $C=C(b)$.
\item[(ii)] If $f\in H_{w}^{1}$, then $A_{D}f\in H_{w,0}^{1}\cap H_{w}^{2}$ and $||A_{D}f||_{2,w}\leq C ||f||_{1,w}$ for some constant $C=C(b)$.
\end{itemize}
\end{lemma}
\begin{remark}
\label{remark21}
The extension of Lemma \ref{lemma23} to weighted, higher-order Sobolev spaces can be obtained from \cite{BernardiM1989}: if $f'\in H_{w}^{m}, m\geq 1$ integer ($f\in H_{w}^{m+1}$) in (\ref{jacob24}), then $v''\in H_{w}^{m}$, being $v=A_{D}f$. Hence, by induction, we have $v\in H_{w}^{m+2}$ and
$$||v||_{m+2,w}\leq C||f||_{m,w}\leq C||f||_{m+1,w},$$ for some constant $C$.
\end{remark}

In the sequel we will make use of the integral formulation of (\ref{BS2})-(\ref{BS2c}) as (cf. \cite{ADM1})
\begin{eqnarray}
\eta(x,t)&=&\mathcal{G}_{1}(x,t,\eta,u)=H(x,t)+\int_{0}^{t}A_{D}(u+\eta u)d\tau,\label{intf1}\\
u(x,t)&=&\mathcal{G}_{2}(x,t,\eta,u)=U(x,t)+\int_{0}^{t}A_{D}\left(c\eta_{xx}+\eta +\frac{u^{2}}{2}\right)d\tau.\label{intf2}
\end{eqnarray}
where
\begin{eqnarray}
H(x,t)&=&\eta_{0}(x)+\frac{w_{2}(x)}{\alpha}\left(h_{1}(t)-h_{1}(0)\right)+\frac{w_{1}(x)}{\alpha}\left(h_{2}(t)-h_{2}(0)\right),\nonumber\\
U(x,t)&=&u_{0}(x)+\frac{w_{2}(x)}{\alpha}\left(v_{1}(t)-v_{1}(0)\right)+\frac{w_{1}(x)}{\alpha}\left(v_{2}(t)-v_{2}(0)\right),\label{intf2b}
\end{eqnarray}
and $\alpha={\rm sinh}(2/\sqrt{b})$.
The problem of uniqueness of solution of (\ref{intf1})-(\ref{intf2b}) is analyzed in the following result.
\begin{proposition}
\label{proposition24} Let $T>0$, $w$ be a weight defined by (\ref{jacobiw}) for some $-1<\mu<1$, $h_{i}, v_{i}\in C[0,T], i=1,2$. The following holds:
\begin{itemize}
\item If $c<0$ and  $\eta_{0}\in H_{w}^{2}, u_{0}\in H_{w}^{1}$, then the system (\ref{intf1}), (\ref{intf2}) has at most one solution $(\eta,u)\in C_{T}^{2}\times C_{T}^{1}$.
\item If $c=0$ and  $\eta_{0}\in H_{w}^{2}, u_{0}\in H_{w}^{2}$, then the system (\ref{intf1}), (\ref{intf2}) has at most one solution $(\eta,u)\in C_{T}^{2}\times C_{T}^{2}$.
\end{itemize}
\end{proposition}
\begin{proof}
Assume first that $c<0$ and let $(\eta_{i},u_{i}), i=1,2$ be two solutions of (\ref{intf1}), (\ref{intf2}). If $\eta=\eta_{1}-\eta_{2}, u=u_{1}-u_{2}$, then we have
\begin{eqnarray*}
\eta(x,t)&=&\int_{0}^{t}A_{D}(u+\eta_{1} u+u_{2}\eta)d\tau,\\
u(x,t)&=&\int_{0}^{t}\left(cA_{D}\eta_{xx}+A_{D}\eta +\frac{1}{2}A_{D}(u(u_{1}+u_{2}))\right)d\tau.
\end{eqnarray*}
Using lemmas \ref{lemma22} and \ref{lemma23}, there is some constant $C>0$ such that
\begin{eqnarray}
||\eta(t)||_{2,w}&\leq & \int_{0}^{t}||A_{D}(u+\eta_{1} u+u_{2}\eta)||_{2,w}d\tau\nonumber\\
&\leq & C\int_{0}^{t}\left(||u(\tau)||_{1,w}+||\eta_{1}(\tau)||_{1,w}||u(\tau)||_{1,w}\right.\nonumber\\
&&\left.+||u_{2}(\tau)||_{1,w}||\eta(\tau)||_{1,w}\right)d\tau\nonumber\\
&\leq & C\left(\left(1+||\eta_{1}||_{C_{T}^{1}}\right)\int_{0}^{t}||u(\tau)||_{1,w}d\tau\right.\nonumber\\
&&\left.+||u_{2}||_{C_{T}^{1}}\int_{0}^{t}||\eta(\tau)||_{1,w}d\tau\right),\label{intf3}
\end{eqnarray}
and
\begin{eqnarray}
||u(t)||_{1,w}&\leq & \int_{0}^{t}\left\|\left(cA_{D}\eta_{xx}+A_{D}\eta +\frac{1}{2}A_{D}(u(u_{1}+u_{2}))\right)\right\|_{1,w}d\tau\nonumber\\
&\leq & C\int_{0}^{t}\left(|c|||\eta_{xx}(\tau)||_{0,w}+||\eta(\tau)||_{0,w}\right.\nonumber\\
&&\left.+\frac{1}{2}||(u_{1}(\tau)+u_{2}(\tau))u(\tau)||_{1,w}\right)d\tau\nonumber\\
&\leq & C\left(\int_{0}^{t}(1+|c|)||\eta(\tau)||_{2,w}d\tau\right.\nonumber\\
&&\left.+\frac{1}{2}||u_{1}+u_{2}||_{C_{T}^{1}}\int_{0}^{t}||u(\tau)||_{1,w}d\tau\right).
\label{intf4}
\end{eqnarray}
Therefore, there is some constant $C>0$ such that for $0\leq t\leq T$
\begin{eqnarray*}
||\eta(t)||_{2,w}+||u(t)||_{1,w}\leq C\int_{0}^{t}\left(||\eta(\tau)||_{2,w}+||u(\tau)||_{1,w}\right)d\tau,
\end{eqnarray*}
from which, by Gronwall's lemma, we have that $\eta=u=0$ and uniqueness follows. In the case $c=0$, we still consider (\ref{intf3}) while, as for (\ref{intf4}), from lemmas \ref{lemma22} and \ref{lemma23} we have
\begin{eqnarray*}
||u(t)||_{2,w}&\leq & \int_{0}^{t}\left\|\left(A_{D}\eta +\frac{1}{2}A_{D}(u(u_{1}+u_{2}))\right)\right\|_{2,w}d\tau\\
&\leq & C\int_{0}^{t}\left(||\eta(\tau)||_{1,w}+\frac{1}{2}||(u_{1}+u_{2})u||_{2,w}\right)d\tau\\
&\leq & C\left(\int_{0}^{t}||\eta(\tau)||_{2,w}d\tau+\frac{1}{2}||u_{1}+u_{2}||_{C_{T}}^{2}\int_{0}^{t}||u(\tau)||_{2,w}d\tau\right).
\end{eqnarray*}
Now
\begin{eqnarray*}
||\eta(t)||_{2,w}+||u(t)||_{2,w}\leq C\int_{0}^{t}\left(||\eta(\tau)||_{2,w}+||u(\tau)||_{2,w}\right)d\tau,
\end{eqnarray*}
for some constant $C>0$ and $0\leq t\leq T$, and the conclusion follows again from Gronwall's lemma.
\end{proof}
The following theorem proves the existence of solution, locally in time, of (\ref{intf1}), (\ref{intf2}) in suitable weighted Sobolev spaces.
\begin{theorem}
\label{theorem25} Let $T>0$, $w$ be a weight defined by (\ref{jacobiw}) for some $-1<\mu<1$, and $h_{i}, v_{i}\in C[0,T], i=1,2$ with the compatibility conditions
\begin{eqnarray*}
&&\eta_{0}(-1)=h_{1}(0),\quad \eta_{0}(1)=h_{2}(0),\\
&&u_{0}(-1)=v_{1}(0),\quad u_{0}(1)=v_{2}(0).
\end{eqnarray*}
Then it holds that:
\begin{itemize}
\item If $c<0$ and  $\eta_{0}\in H_{w}^{2}, u_{0}\in H_{w}^{1}$, there exists $T_{0}=T_{0}(T,||\eta_{0}||_{2,w},||u_{0}||_{1,w})$ such that the system (\ref{intf1}), (\ref{intf2}) has a unique solution $(\eta,u)\in C_{T_{0}}^{2}\times C_{T_{0}}^{1}$ satisfying
\begin{eqnarray}
||\eta||_{C_{T_{0}}^{2}}+||\eta_{t}||_{C_{T_{0}}^{2}}\leq C,\quad
||u||_{C_{T_{0}}^{1}}+||u_{t}||_{C_{T_{0}}^{1}}\leq C,\label{l210}
\end{eqnarray}
for some constant $C=C(T_{0})$.
\item If $c=0$ and  $\eta_{0}\in H_{w}^{2}, u_{0}\in H_{w}^{2}$, there exists $T_{0}=T_{0}(T,||\eta_{0}||_{2,w},||u_{0}||_{2,w})$ such that the system (\ref{intf1}), (\ref{intf2}) has a unique solution $(\eta,u)\in C_{T_{0}}^{2}\times C_{T_{0}}^{2}$ satisfying
\begin{eqnarray}
||\eta||_{C_{T_{0}}^{2}}+||\eta_{t}||_{C_{T_{0}}^{2}}\leq C,\quad
||u||_{C_{T_{0}}^{2}}+||u_{t}||_{C_{T_{0}}^{2}}\leq C,\label{l211}
\end{eqnarray}
for some constant $C=C(T_{0})$.
\end{itemize}
\end{theorem}
\begin{proof}
The proof is similar to that of Proposition 2.2 in \cite{ADM1}. 
Let us first assume that $c<0$. Let $T_{0}>0$ (to be specified later) and consider the Banach space $E=C_{T_{0}}^{2}\times C_{T_{0}}^{1}$ with norm (\ref{Enorm}),
and the mapping $\Gamma:E\rightarrow E$
\begin{eqnarray*}
\Gamma(\alpha,\beta)(t)=\begin{pmatrix} H+\int_{0}^{t}A_{D}((1+\alpha)\beta)d\tau\\U+\int_{0}^{t}\left(cA_{D}\alpha_{xx}+A_{D}\left(\alpha+\frac{\beta^{2}}{2}\right)\right)d\tau\end{pmatrix},
\end{eqnarray*}
for $0\leq t\leq T_{0}$ and where $H=H(x,t), U=U(x,t)$ are given by (\ref{intf2b}). Let $B_{R}$ be the closed ball in $E$ centered at zero and with radius $R>0$. Let $(\eta_{i},u_{i})\in B_{R}, i=1,2$ and $\eta=\eta_{1}-\eta_{2}, u=u_{1}-u_{2}$. By using again the estimates obtained in Proposition \ref{proposition24} for uniqueness, we have
\begin{eqnarray*}
\left|\left|\int_{0}^{t}\left(A_{D}((1+\eta_{1})u_{1})-A_{D}((1+\eta_{2})u_{2})\right)d\tau\right|\right|_{2,w}=\left|\left|\int_{0}^{t}A_{D}(u+\eta_{1}u+u_{2}\eta)d\tau\right|\right|_{2,w}&&\\
\leq  C_{1}\int_{0}^{t}\left((1+||\eta_{1}(\tau)||_{1,w})||u(\tau)||_{1,w}+||u_{2}(\tau)||_{1,w}||\eta(\tau)||_{1,w}\right)d\tau,&&
\end{eqnarray*}
and
\begin{eqnarray*}
\left|\left|\int_{0}^{t}\left(cA_{D}\eta_{1xx}+A_{D}\left(\eta_{1}+\frac{u_{1}^{2}}{2}\right)-
cA_{D}\eta_{2xx}+A_{D}\left(\eta_{2}+\frac{u_{2}^{2}}{2}\right)\right)d\tau\right|\right|_{1,w}&&\\
=\left|\left|\int_{0}^{t}\left(cA_{D}\eta_{xx}+A_{D}\eta
+\frac{1}{2}A_{D}\left((u_{1}+u_{2})u\right)\right)d\tau\right|\right|_{1,w}&&\\
\leq C_{2}\left((1+|c|)\int_{0}^{t}||\eta(\tau)||_{2,w}+\frac{1}{2}\int_{0}^{t}||u_{1}+u_{2}||_{1,w}||u(\tau)||_{1,w}d\tau\right),
\end{eqnarray*}
for some constants $C_{1}, C_{2}>0$. Therefore
\begin{eqnarray*}
||\Gamma(\eta_{1},u_{1})-\Gamma(\eta_{2},u_{2})||_{E}&\leq & C_{1}T_{0}\left((1+||\eta_{1}||_{C_{T_{0}}^{1}})||u_{1}-u_{2}||_{C_{T_{0}}^{1}}\right.\\
&&\left.+||u_{2}||_{C_{T_{0}}^{1}}||\eta_{1}-\eta_{2}||_{C_{T_{0}}^{1}}\right)\\
&&+C_{2}T_{0}\left((1+|c|)||\eta_{1}-\eta_{2}||_{C_{T_{0}}^{2}}\right.\\
&&\left.+\frac{1}{2}||u_{1}+u_{2}||_{C_{T_{0}}^{1}}||u_{1}-u_{2}||_{C_{T_{0}}^{1}}\right)\\
&\leq & \widetilde{C}(T_{0})||(\eta_{1},u_{1})-(\eta_{2},u_{2})||_{E},
\end{eqnarray*}
where 
\begin{eqnarray*}
\widetilde{C}(T_{0}))=T_{0}(\underbrace{C_{1}(1+2R)+C_{2}(1+|c|+R)}_{\overline{C}}).
\end{eqnarray*}
Moreover, if $(\eta,u)\in B_{R}$ then
\begin{eqnarray*}
||\Gamma(\eta,u)||_{E}&\leq & ||\Gamma(\eta,u)-\Gamma(0,0)||_{E}+||\Gamma(0,0)||_{E}\\
&\leq & \widetilde{C}(T_{0})R+||(H,U)||_{E},
\end{eqnarray*}
and
\begin{eqnarray*}
||(H,U)||_{E}&=&||H||_{C_{T_{0}}^{2}}+||U||_{C_{T_{0}}^{1}}\\
&=&\sup_{0\leq t\leq T_{0}}||H(t)||_{2,w}+\sup_{0\leq t\leq T_{0}}||U(t)||_{1,w},
\end{eqnarray*}
with, for $0\leq t\leq T$
\begin{eqnarray*}
||H(t)||_{2,w}&\leq & ||\eta_{0}||_{2,w}+\frac{|h_{1}(t)-h_{1}(0)|}{\alpha}||w_{2}||_{2,w}\\
&&+\frac{|h_{2}(t)-h_{2}(0)|}{\alpha}||w_{1}||_{2,w},
\end{eqnarray*}
where $\alpha={\rm sinh}(2/\sqrt{b})$. Then
\begin{eqnarray*}
\sup_{0\leq t\leq T_{0}}||H(t)||_{2,w}&\leq & ||\eta_{0}||_{2,w}+\frac{2}{\alpha}||w_{2}||_{2,w}\sup_{0\leq t\leq T}|h_{1}(t)|\\
&&+\frac{2}{\alpha}||w_{1}||_{2,w}\sup_{0\leq t\leq T}|h_{2}(t)|=:G_{0},
\end{eqnarray*}
and similarly it can be shown that
\begin{eqnarray*}
\sup_{0\leq t\leq T_{0}}||U(t)||_{1,w}&\leq & ||u_{0}||_{1,w}+\frac{2}{\alpha}||w_{2}||_{1,w}\sup_{0\leq t\leq T}|v_{1}(t)|\\
&&+\frac{2}{\alpha}||w_{1}||_{1,w}\sup_{0\leq t\leq T}|v_{2}(t)|=:G_{1}.
\end{eqnarray*}
If we take $\gamma_{0}=G_{0}+G_{1}$, then 
\begin{eqnarray*}
||\Gamma(\eta,u)||_{E}\leq \widetilde{C}(T_{0})R+\gamma_{0}.
\end{eqnarray*}
We choose $R=2\gamma_{0}$ and $T_{0}\leq \min\{T,\frac{1}{2\overline{C}}\}$. Then $\widetilde{C}(T_{0})R+\gamma_{0}=\overline{C}T_{0}R+\gamma_{0}\leq R$ and the contraction mapping theorem applies to $\Gamma$ as a mapping $\Gamma:B_{R}\rightarrow B_{R}$. Thus, $\Gamma$ has a unique fixed point $(\eta,u)\in B_{R}$, which is the solution of (\ref{intf1}), (\ref{intf2}). Uniqueness comes from Proposition \ref{proposition24}. Furthermore, from the representation of (\ref{BS2})
\begin{eqnarray*}
\eta_{t}&=&A_{D}\left((1+\eta)u\right),\\
u_{t}&=&A_{D}\left(c\eta_{xx}+\eta+\frac{u^{2}}{2}\right),
\end{eqnarray*}
with $A_{D}$ given by (\ref{jacob23b}), and lemma \ref{lemma23}, we have, for $0\leq t\leq T_{0}$
\begin{eqnarray*}
||\eta_{t}(t)||_{2,w}&\leq & C||(1+\eta(t))u(t)||_{1,w},\\
||u_{t}(t)||_{1,w}&\leq & C||c\eta_{xx}(t)+\eta(t)+\frac{u(t)^{2}}{2}||_{0,w},
\end{eqnarray*}
for some constant $C$, which completes the proof of (\ref{l210}).

When $c=0$, we must consider $E=C_{T_{0}}^{2}\times C_{T_{0}}^{2}$ with the corresponding norm given by (\ref{Enorm}),
and the mapping $\Gamma:E\rightarrow E$
\begin{eqnarray*}
\Gamma(\alpha,\beta)(t)=\begin{pmatrix} H+\int_{0}^{t}A_{D}((1+\alpha)\beta)d\tau\\U+\int_{0}^{t}\left(A_{D}\left(\alpha+\frac{\beta^{2}}{2}\right)\right)d\tau\end{pmatrix},
\end{eqnarray*}
and it is not hard to adapt the arguments above in order to obtain (\ref{l211}).
\end{proof}
\begin{remark}
\label{remark23}
As in \cite{ADM1}, higher regularity of the solution can be derived from the hypothesis of more regular data. Using Remark \ref{remark21}, we can also extend the previous result in the following way: when $h_{i}, v_{i}\in C^{l}([0,T])$ with $l\geq 1$, then there holds:
\begin{itemize}
\item If $c<0$ and $\eta_{0}\in H_{w}^{m+2}, u_{0}\in H_{w}^{m+1}, m\geq 1$, then the solution of (\ref{intf1}), (\ref{intf2}) satisfies $(\eta,u)\in C_{T_{0}}^{m+2}\times C_{T_{0}}^{m+1}$ with $(\eta_{t}^{k},u_{t}^{k})\in C_{T_{0}}^{m+2}\times C_{T_{0}}^{m+1}$ for $1\leq k\leq l$.
\item If $c=0$ and $\eta_{0}\in H_{w}^{m+2}, u_{0}\in H_{w}^{m+2}, m\geq 1$, then the solution of (\ref{intf1}), (\ref{intf2}) satisfies $(\eta,u)\in C_{T_{0}}^{m+2}\times C_{T_{0}}^{m+2}$ with $(\eta_{t}^{k},u_{t}^{k})\in C_{T_{0}}^{m+2}\times C_{T_{0}}^{m+2}$ for $1\leq k\leq l$.
\end{itemize}
\end{remark}

\section{Numerical approximation}
\label{sec3}
In \cite{ADM2}, the ibvp (\ref{BS2}) is approximated using the standard Galerkin-finite element method for the spatial discretization and a fourth-order explicit RK scheme as time integrator. The present paper proposes a different numerical treatment, approximating the ibvp in space with a spectral Galerkin method based on Jacobi polynomials while the time integration is chosen according to the preservation of some stability issues that may improve the performance of the full discretization in long time simulations and with nonregular data.

\subsection{Weak formulation of the homogeneous problem}
The description and analysis of the spectral discretizations will be done by considering a related problem with homogeneous boundary conditions. Assume that $h_{i}, v_{i}\in C^{1}, i=1,2$. The functions
\begin{eqnarray*}
\overline{\eta}(x,t)&=&\left(\frac{h_{2}(t)-h_{1}(t)}{2}\right)(x+1)+h_{1}(t),\\
\overline{u}(x,t)&=&\left(\frac{v_{2}(t)-v_{1}(t)}{2}\right)(x+1)+v_{1}(t),
\end{eqnarray*}
satisfy (\ref{BS2b}), (\ref{BS2c}). Let $\eta, u$ be the solution of (\ref{BS2}). Then $\widetilde{\eta}=\eta-\overline{\eta}, \widetilde{u}=u-\overline{u}$  solve a problem of the form (tildes are dropped)
\begin{eqnarray}
&&\eta_{t}+((1+\overline{\eta}+\eta) u)_{x}+(\overline{u}\eta)_{x}-b\eta_{xxt}+\alpha(x,t)=0,\quad x\in\overline{\Omega}, \;t>0,\nonumber\\
&&u_{t}+\eta_{x}+(\overline{u}u)_{x}+uu_{x}+c\eta_{xxx}-bu_{xxt}+\beta(x,t)=0,\label{BS2d}\\
&&\eta(x,0)=\eta_{0}(x)-\overline{\eta}(x,0)=:\widetilde{\eta}_{0}(x), \quad u(x,0)=u_{0}(x)-\overline{u}(x,0)=\widetilde{u}_{0}(x),\quad x\in\overline{\Omega},\nonumber\\
&&\eta(\pm 1,t)=u(\pm 1,t)=0,\quad t>0,\nonumber
\end{eqnarray}
and 
\begin{eqnarray*}
\alpha&=&(\overline{\eta}\overline{u})_{x}+\overline{\eta}_{t}+\overline{u}_{x}=\partial_{x}\widetilde{\alpha}, \\
\beta&=&\overline{u}\overline{u}_{x}+\overline{u}_{t}+\overline{\eta}_{x}=\partial_{x}\widetilde{\beta},
\end{eqnarray*}
with

\begin{eqnarray}
\widetilde{\alpha}(x,t)&=&\alpha_{1}(t)\frac{(x+1)^{2}}{2}+\alpha_{2}(t)(x+1),\nonumber\\
\widetilde{\beta}(x,t)&=&\beta_{1}(t)\frac{(x+1)^{2}}{2}+\beta_{2}(t)(x+1),\label{BS2e}\\
\alpha_{1}(t)&=&\frac{(h_{2}(t)-h_{1}(t))(v_{2}(t)-v_{1}(t))}{2}+\frac{(h_{2}'(t)-h_{1}'(t))}{2},\nonumber\\
 \alpha_{2}(t)&=&\frac{(h_{2}(t)-h_{1}(t))}{2}v_{1}(t)+\frac{(v_{2}(t)-v_{1}(t))}{2}(1+h_{1}(t))\nonumber\\
&&+h_{1}'(t)+\frac{(v_{2}(t)-v_{1}(t))}{2},\nonumber\\
\beta_{1}(t)&=&\left(\frac{(v_{2}(t)-v_{1}(t))}{2}\right)^{2}+\frac{(v_{2}'(t)-v_{1}'(t))}{2},\nonumber\\
\beta_{2}(t)&=&\frac{(v_{2}(t)-v_{1}(t))}{2}v_{1}(t)+\frac{(v_{2}(t)-v_{1}(t))}{2}v_{1}(t)\nonumber\\
&&+v_{1}'(t)+\frac{(h_{2}(t)-h_{1}(t))}{2}.\nonumber
\end{eqnarray}
Now, the problem (\ref{BS2d}), (\ref{BS2e}) admits the following weak formulation
\begin{eqnarray*}
A_{b}(\eta_{t},\psi)+B_{1}(\eta,u,\psi)+(\alpha,\psi)_{w}&=&0,\label{BS24a}\\
A_{b}(u_{t},\phi)+B_{2}(\eta,u,\phi)+(\beta,\phi)_{w}&=&0,\label{BS24b}
\end{eqnarray*}
for $\phi, \psi\in H_{w,0}^{1}$, where $A_{b}$ is given by (\ref{ad26}) and
\begin{eqnarray*}
B_{1}(\eta,u,\psi)&=&((1+\overline{\eta})u+\overline{u}\eta+\eta u)_{x},\psi)_{w},\nonumber\\
B_{2}(\eta,u,\phi)&=&((\eta+\overline{u}u)_{x}+uu_{x},\phi)_{w}+|c|L(\eta_{xx},\phi),\label{BS25}
\end{eqnarray*}
with $L$ given by (\ref{ad26a}).
\subsection{Spectral Galerkin approximation}
In this section the spectral Galerkin approximation to (\ref{BS2}) will be analyzed via the semidiscretization of the homogeneous problem (\ref{BS2d}) as follows. We first make use of the elliptic projection $R_{N}v\in\mathbb{P}_{N}^{0}$ of $v\in H_{w,0}^{1}$ defined in (\ref{ad27}). Note that, evaluating (\ref{ad27}) at $\varphi=R_{N}v$ and using ellipticity and continuity of $A_{b}$ there holds
\begin{eqnarray}
||R_{N}v||_{1,w}\leq C||v||_{1,w},\label{l36}
\end{eqnarray}
for some constant $C$. In addition, assume that $v\in H_{w}^{2}\cap H_{w,0}^{1}$. Then, from (\ref{ad27}) with $\varphi\in H_{w,0}^{1}$
\begin{eqnarray*}
A_{b}(R_{N}v,\varphi)=A_{b}(v,\varphi)=(v,\varphi)_{w}+bL(v,\varphi)=(v-bv'',\varphi)_{w}.
\end{eqnarray*}
Lax-Milgram theorem implies that $R_{N}v\in H_{w}^{2}$ and
\begin{eqnarray}
||R_{N}v||_{2,w}\leq C||v-bv''||_{0,w}\leq C||v||_{2,w},\label{l37}
\end{eqnarray}
for some constant $C$.

Let $-1<\mu<1$ and $w$ be the corresponding weight defined in (\ref{jacobiw}). Let $[0,T]$ be the maximal interval of existence and uniqueness of solution of (\ref{BS2}) (or, equivalently, (\ref{BS2d}) ) in $C_{T}^{2}\times C_{T}^{1}$ if $c<0$ and $C_{T}^{2}\times C_{T}^{2}$ if $c=0$.
Let $N\geq 2$ be an integer. The semidiscrete Galerkin approximation to (\ref{BS2d}) is defined as the pair $\eta^{N}, u^{N}:[0,T]\rightarrow\mathbb{P}_{N}^{0}$ satisfying for any $\psi, \phi\in\mathbb{P}_{N}^{0}$ and $0\leq t\leq T$
\begin{eqnarray}
A_{b}(\eta_{t}^{N},\psi)+B_{1}(\eta^{N},u^{N},\psi)+(\alpha,\psi)_{w}&=&0,\label{BS26a}\\
A_{b}(u_{t}^{N},\phi)+B_{2}(\eta^{N},u^{N},\phi)+(\beta,\phi)_{w}&=&0,\label{BS26b}
\end{eqnarray}
with
\begin{eqnarray}
\eta^{N}(0)=R_{N}\widetilde{\eta}_{0},\quad
u^{N}(0)=R_{N}\widetilde{u}_{0}.\label{BS26c}
\end{eqnarray}
An alternative formulation of (\ref{BS26a})-(\ref{BS26c}) will be used below in the proof of existence and the error estimates, cf. \cite{ADM2}. We consider the restriction of $A_{b}$ to $H_{w,0}^{1}\times H_{w,0}^{1}$ (denoted again for simplicity as $A_{b}$) and defined, via Lax-Milgram theorem, $\widehat{f}:L_{w}^{2}\rightarrow \mathbb{P}_{N}^{0}$ satisfying, for any $\varphi\in \mathbb{P}_{N}^{0}$
\begin{eqnarray}
A_{b}(\widehat{f}(v),\varphi)=-(v',\varphi)_{w}=\int_{-1}^{1}v(\varphi w)_{x}dx.\label{l310} 
\end{eqnarray}
From (\ref{l310}) let
\begin{eqnarray}
&&\widehat{g}:H_{w}^{2}\rightarrow \mathbb{P}_{N}^{0},\quad \widehat{g}(v)=c\widehat{f}(v'')+\widehat{f}(v),\nonumber\\
&&F:H_{w}^{1}\times H_{w}^{1}\rightarrow \mathbb{P}_{N}^{0},\quad F(v,w)=\widehat{f}(vw).\label{l311}
\end{eqnarray}
Using (\ref{l311}) we define $f,g:H_{w}^{1}\times H_{w}^{1}\rightarrow \mathbb{P}_{N}^{0}$
\begin{eqnarray}
f(\eta,u)&=&F(1+\overline{\eta},u)+F(\eta,\overline{u})+F(\eta,u)+\widehat{f}(\widetilde{\alpha}),\nonumber\\
g(\eta,u)&=&\widetilde{g}(\eta)+\frac{1}{2}F(u,u)+F(\overline{u},u)+\widehat{f}(\widetilde{\beta}),\label{l312}
\end{eqnarray}
and write the semidiscrete problem (\ref{BS26a})-(\ref{BS26c}) as
\begin{eqnarray}
&&\partial_{t}\eta^{N}=f(\eta^{N},u^{N}),\quad \eta^{N}(0)=R_{N}\widetilde{\eta}_{0},\nonumber\\
&&\partial_{t}u^{N}=g(\eta^{N},u^{N}),\quad u^{N}(0)=R_{N}\widetilde{u}_{0}.\label{l313}
\end{eqnarray}
If we represent $\eta^{N}, u^{N}$ in a basis of $\mathbb{P}_{N}^{0}$ then (\ref{l313}) yields an ode system for the corresponding coefficients. Continuity and locally Lipschitz property of (\ref{l312}) ensures local existence of solution $(\eta^{N}, u^{N})$ of (\ref{l313}). The extension of solution to the final time $t=T$ requires additional properties of $f$ and $g$, proved in the following lemma.
\begin{lemma}
\label{lemma31} There is a constant $C>0$ such that
\begin{eqnarray}
||\widehat{f}(v)||_{1,w}&\leq &C||v||_{0,w},\quad v\in L_{w}^{2},\label{l314a}\\
||\widehat{f}(v)||_{2,w}&\leq &C||v||_{1,w},\quad v\in L_{w}^{2}.\label{l314b}
\end{eqnarray}
\end{lemma}
\begin{proof}
The inequality (\ref{l314a}) is obtained from the second form of the right-hand side of (\ref{l310}), evaluated at $\varphi=\widehat{f}(v)$, ellipticity of $A_{b}$ on $H_{w,0}^{1}\times H_{w,0}^{1}$, and continuity of $L$ defined in (\ref{ad26a}). On the other hand, (\ref{l314b}) follows from the application of Lax-Milgram theorem to (\ref{l310}) with the first form of its right-hand side.
\end{proof}
The following theorem is the main result of existence and convergence of the semidiscrete approximation in the norm of suitable weighted Sobolev spaces.
\begin{theorem}
\label{theorem32} Let $T>0$, $N\geq 1$ be an integer, $m>3$ and let $(\eta,u)\in E$ be a solution of (\ref{BS2d}), (\ref{BS2e}) with
\begin{eqnarray*}
E=\left\{\begin{matrix}C_{T}^{m}\times C_{T}^{m-1}&c<0\\C_{T}^{m}\times C_{T}^{m}&c=0\end{matrix}\right.
\end{eqnarray*}
Then  there is a unique solution $(\eta^{N},u^{N})$ of (\ref{l313}) on $[0,T]$ with
\begin{eqnarray}
||(\eta^{N},u^{N})||_{E}\leq C,\label{l315}
\end{eqnarray}
for some constant $C$.
Furthermore, there is constant $C>0$ such that for $0\leq t\leq T$
\begin{eqnarray}
||\eta(t)-\eta^{N}(t)||_{2,w}+||u(t)-u^{N}(t)||_{1,w}\leq C N^{3-m},\label{l315a}
\end{eqnarray}
if $c<0$ and
\begin{eqnarray}
||\eta(t)-\eta^{N}(t)||_{2,w}+||u(t)-u^{N}(t)||_{2,w}\leq C N^{1-m},\label{l315b}
\end{eqnarray}
if $c=0$.
\end{theorem}
\begin{proof}
We follow and adapt the lines of the proof of Proposition 5 in \cite{ADM2}. We define
\begin{eqnarray*}
\rho=\eta-R_{N}\eta,\quad \theta=R_{N}\eta-\eta^{N},\quad \sigma=u-R_{N}u,\quad \xi=R_{N}u-u^{N}.
\end{eqnarray*}
Then $\eta-\eta^{N}=\rho+\theta$ and $u-u^{N}=\xi+\sigma$. We first assume $c<0$. Using theorem \ref{theorem25}, (\ref{l36}), and (\ref{l37}), let $M>0$ be a constant such that
\begin{eqnarray}
\max_{0\leq t\leq T}\left(||\eta(t)||_{2,w}+||u(t)||_{1,w}\right)\leq M,\label{l316a}
\end{eqnarray}
and
\begin{eqnarray}
||R_{N}\widetilde{\eta}_{0}||_{2,w}+||R_{N}\widetilde{u}_{0}||_{1,w}\leq \frac{3M}{2}.\label{l316c}
\end{eqnarray}
Let $0<t_{N}\leq T$ be the maximal time for which
\begin{eqnarray*}
||u^{N}(t)||_{1,w}\leq 2M,\quad 0\leq t\leq t_{N}.
\end{eqnarray*}
From (\ref{l313}) and (\ref{BS26a}), (\ref{BS26b}) it holds that
\begin{eqnarray*}
\theta_{t}&=&F(1+\overline{\eta},\xi+\sigma)+F(\overline{u},\rho+\theta)+F(\eta,\xi+\sigma)+F(\rho+\theta,u^{N})\nonumber\\
\xi_{t}&=&\widehat{g}(\rho+\theta)+\frac{1}{2}\left(F(u,\xi+\sigma)+F(\sigma+\xi,u^{N})\right)+F(\overline{u},\xi+\sigma).\label{l316}
\end{eqnarray*}
Now we make use of (\ref{l314a}), (\ref{l314b}), (\ref{ad27b}), and (\ref{ad27d}) to have, for $0\leq t\leq t_{N}$
\begin{eqnarray}
||\theta_{t}(t)||_{2,w}&\leq & C\left(N^{1-m}+||\xi(t)||_{1,w}+||\theta(t)||_{1,w}\right),\nonumber\\
||\xi_{t}(t)||_{1,w}&\leq & C\left(N^{3-m}+||\xi(t)||_{1,w}+||\theta(t)||_{2,w}\right).\label{l316b}
\end{eqnarray}
Therefore, for $0\leq t\leq t_{N}$
\begin{eqnarray*}
||\theta(t)||_{2,w}+||\xi(t)||_{1,w}&\leq &||\theta(0)||_{2,w}+||\xi(0)||_{1,w}\\
&&+C\left(\int_{0}^{t}\left(||\theta(\tau)||_{2,w}+||\xi(\tau)||_{1,w}\right)d\tau+N^{3-m}\right).
\end{eqnarray*}
Using the fact that $\theta(0)=\xi(0)=0$ and Gronwall's lemma, there is a constant $C=C(\eta,u,M,T)$ such that
\begin{eqnarray}
||\theta(t)||_{2,w}+||\xi(t)||_{1,w}\leq C N^{3-m},\quad 0\leq t\leq t_{N}.\label{l317}
\end{eqnarray}
Now, due to (\ref{ad27b}), for $0\leq t\leq t_{N}$ there holds
$$||\sigma(t)||_{1,w}\leq C N^{1-m},$$ and since $u^{N}=u-(\xi+\sigma)$, then
\begin{eqnarray*}
||u^{N}(t)||_{1,w}\leq C N^{3-m}+M,
\end{eqnarray*}
and therefore, for $N$ large enough and since $m>3$, we have $||u^{N}(t)||_{1,w}\leq 2M$ for $0\leq t\leq t_{N}$, which contradicts the definition of $t_{N}$. Hence (\ref{l316a}) and (\ref{l317}) hold up to $t=T$, implying (\ref{l315}), and the estimate (\ref{l315a}) holds by using (\ref{l317}), (\ref{ad27b}), and (\ref{ad27d}).

For the case $c=0$, note that $\widehat{g}(v)=\widehat{f}(v)$. Instead of (\ref{l316a}), (\ref{l316c}), we take $M>0$ such that
\begin{eqnarray*}
\max_{0\leq t\leq T}\left(||\eta(t)||_{2,w}+||u(t)||_{2,w}\right)\leq M,
\end{eqnarray*}
and
\begin{eqnarray*}
||R_{N}\widetilde{\eta}_{0}||_{2,w}+||R_{N}\widetilde{u}_{0}||_{2,w}\leq \frac{3M}{2}.
\end{eqnarray*}
Let $0<t_{N}\leq T$ be now the maximal time for which
\begin{eqnarray*}
||u^{N}(t)||_{2,w}\leq 2M,\quad 0\leq t\leq t_{N}.
\end{eqnarray*}
Then, (\ref{l316b}) is clearly substituted by
\begin{eqnarray*}
||\theta_{t}(t)||_{2,w}&\leq & C\left(N^{1-m}+||\xi(t)||_{1,w}+||\theta(t)||_{1,w}\right),\nonumber\\
||\xi_{t}(t)||_{2,w}&\leq & C\left(N^{1-m}+||\xi(t)||_{1,w}+||\theta(t)||_{2,w}\right),
\end{eqnarray*}
where (\ref{l314b}) and (\ref{ad27b}) are only required. The rest of the argument is similar to that of the case $c<0$, leading to (\ref{l315b}).
\end{proof}
\begin{remark}
Note that, when $c<0$, the presence of the term $c\eta_{xxx}$ in (\ref{BS2}) requires the use of the generalized estimate (\ref{ad27d}) for $m>3$ in order to obtain (\ref{l315a}).
\end{remark}
\subsection{Galerkin-Numerical Integration (G-NI) formulation}
\label{sec33}
Described here is the representation used in practice of the semidiscrete Galerkin approximation and its relation with the spectral collocation approach. The representation is based on two main properties: the choice of a suitable basis of $\mathbb{P}_{N}$ and the quadrature formulas for approximating the resulting integrals. Both are somehow related in the following sense. Some features of (\ref{BS2}) (mainly the presence of nonlinear terms) typically require the use of a basis of $\mathbb{P}_{N}$ of nodal type, which implies the choice of certain set of nodes $\{x_{j}\}_{j=0}^{N}$ in $[-1,1]$ and the construction of interpolating polynomials $\psi_{j}(x), j=0,\ldots,N$ such that
\begin{eqnarray}
\psi_{j}(x_{k})=\delta_{jk},\quad j,k=0,\ldots,N.\label{l319b}
\end{eqnarray}
(Other choices, such as bases of modal tyoe, are not discarded, although they seem to give a better performance in the linear case, cf. \cite{CanutoHQZ1988,ShenTW2011}.) Then the selection of the $x_{j}$ as the nodes of quadrature formulas used to approximate the integrals leads to a more convenient description of the semidiscrete system.

This general procedure is now described in more detail for the case at hand. Let $w(x)$ be a weight (\ref{jacobiw}) for $-1<\mu<1$ and let $\{J_{n}\}_{n}$ be the family of the corresponding Jacobi polynomials associated to $w$ (that is, orthogonal to each other in $L_{w}^{2}$). Among several possibilities, the following Gauss-Lobatto-Jacobi formula is used as quadrature (the proof can be seen in e.~g. \cite{BernardiM1997}).
\begin{theorem}
\label{theorem198}
Let $N$ be a positive integer and $\{x_{j}\}_{j=0}^{N}$ be given by $x_{0}=-1, x_{N}=1$ and $\{x_{j}\}_{j=1}^{N-1}$ as the zeros of $J_{N}'(x)$ in $(-1,1)$. Then there exists a unique set of $N+1$ real numbers $w_{j}, j=0,\ldots,N$, such that for any $p\in\mathbb{P}_{2N-1}(-1,1)$ it holds that
\begin{eqnarray}
\int_{-1}^{1}p(x)w(x)dx=\sum_{j=0}^{N}p(x_{j})w_{j}.\label{l320}
\end{eqnarray}
\end{theorem}

Note that (\ref{l320}) means that the quadrature rule, for integrals with weight $w$, from the set $\{(x_{j},w_{j})\}_{j=0}^{N}$ is exact for polynomials of degree less than or equals $2N-1$.

The nodal basis $\{\psi_{j}\}_{j=0}^{N}$ of $\mathbb{P}_{N}$ based on the $x_{j}$ and considered here is explicitly given by
\begin{eqnarray}
\psi_{j}(x)=C(N,\mu)\frac{(1-x^{2})J_{N}'(x)}{(x_{j}-x)J_{N}(x_{j})},\quad j=0,\ldots,N,\label{l321}
\end{eqnarray}
where
\begin{eqnarray*}
C(N,\mu)=\left\{\begin{matrix} \frac{1}{N(N+2\mu+1)}&1\leq j\leq N-1\\
\frac{\mu+1}{N(N+2\mu+1)}&j=0,N\end{matrix}.\right.
\end{eqnarray*}
The property (\ref{l319b}) for the family (\ref{l321}) follows from the differential equation satisfied by the Jacobi polynomial $J_{N}$, \cite{BernardiM1997}
\begin{eqnarray}
\frac{d}{dx}\left((1-x^{2})J_{N}'(x)w(x)\right)+N(N+2\mu+1)w(x)J_{N}(x)=0.\label{l322}
\end{eqnarray}
We now describe the G-NI formulation to discretize the original problem (\ref{BS2})-(\ref{BS2c}) with the spectral Galerkin method based on the $J_{n}$. (Recall Remark \ref{remark20} as for the problem on general intervals $[a,b]$ and the previous change to $[-1,1]$.) We first represent the approximations $\eta^{N}, u^{N}$ in the nodal basis (\ref{l321})
\begin{eqnarray}
\eta^{N}(x,t)=\sum_{k=0}^{N}\eta_{k}(t)\psi_{k}(x),\quad
u^{N}(x,t)=\sum_{k=0}^{N}u_{k}(t)\psi_{k}(x).\label{l323}
\end{eqnarray}
Note that the nodal type of (\ref{l321}) implies that
$$\eta_{k}(t)=\eta^{N}(x_{k},t),\quad u_{k}(t)=u^{N}(x_{k},t),\quad k=0,\ldots,N,\quad t\geq 0.$$ In particular, this
allows to include the boundary conditions in the representation (\ref{l323}) directly as
\begin{eqnarray}
\eta_{0}(t)=h_{1}(t),\; \eta_{N}(t)=h_{2}(t),\; u_{0}(t)=v_{1}(t),\; u_{N}(t)=v_{2}(t).\label{l323b}
\end{eqnarray}
The unknowns $\eta_{k}, u_{k}, k=1,\ldots,N-1$ are determined from an ode system obtained by substituting (\ref{l323}) into (\ref{BS2}) and making the weighted inner product (\ref{12a}) of each equation with $\psi_{i}, i=1,\ldots,N-1$. This leads to
\begin{eqnarray}
&&\sum_{k=0}^{N}\eta_{k}'(t)\left((\psi_{k},\psi_{i})_{w}-b(\psi_{k}'',\psi_{i})_{w}\right)+u_{k}(t)(\psi_{k}',\psi_{i})_{w}\nonumber\\
&&+(\partial_{x}(\eta^{N}u^{N}),\psi_{i})_{w}=0,\label{l324a}\\
&&\sum_{k=0}^{N}u_{k}'(t)\left((\psi_{k},\psi_{i})_{w}-b(\psi_{k}'',\psi_{i})_{w}\right)+\eta_{k}(t)\left((\psi_{k}',\psi_{i})_{w}+c(\psi_{k}''',\psi_{i})_{w}\right)\nonumber\\
&&+\frac{1}{2}(\partial_{x}(u^{N})^{2},\psi_{i})_{w}=0.\label{l324b}.
\end{eqnarray}
The contribution of the terms (\ref{l323b}) in (\ref{l324a}), (\ref{l324b}) will be translated to the right-hand side as known data in the final process. The G-NI formulation consists now of approximating (or, in some cases and according to theorem \ref{theorem198}, computing exactly) the integrals appearing in (\ref{l324a}), (\ref{l324b}) with the quadrature (\ref{l320}) and making use of the property (\ref{l319b}). The computation of each specific case is described below. Let $k=0,\ldots N$ and $i=1,\ldots,N-1$. Then

\begin{eqnarray}
(\psi_{k},\psi_{i})_{w}=\int_{-1}^{1}\psi_{k}(x)\psi_{i}(x)w(x)dx\sim\sum_{h=0}^{N}\psi_{k}(x_{h})\psi_{i}(x_{h})w_{h}=w_{i}\delta_{ki}.\label{l324ab}
\end{eqnarray}
Note that the integration is exact for
\begin{eqnarray*}
(\psi_{k}',\psi_{i})_{w}=\int_{-1}^{1}\psi_{k}'(x)\psi_{i}(x)w(x)dx=\sum_{h=0}^{N}\psi_{k}'(x_{h})\psi_{i}(x_{h})w_{h}=\psi_{k}'(x_{i})w_{i}.
\end{eqnarray*}
On the other hand, we integrate by parts to have
\begin{eqnarray}
(\psi_{k}'',\psi_{i})_{w}&=&\int_{-1}^{1}\psi_{k}''(x)\psi_{i}(x)w(x)dx=\psi_{k}'(x)\psi_{i}(x)w(x)\big\|_{-1}^{1}\nonumber\\
&&-\int_{-1}^{1}\psi_{k}'(x)\left(\psi_{i}w\right)'(x)dx.\label{l324c}
\end{eqnarray}
Note that, due to the form (\ref{l321}) of $\psi_{i}$ for $i=1,\ldots,N-1$ and since $-1<\mu<1$, then the first term on the right-hand side of (\ref{l324c}) vanishes. For the second term, we make use of the property
$$w'(x)=\frac{-2x\mu}{1-x^{2}}w(x),$$ and write
\begin{eqnarray}
\int_{-1}^{1}\psi_{k}'(x)\left(\psi_{i}w\right)'(x)dx=\int_{-1}^{1}\psi_{k}'(x)\psi_{i}'(x)w(x)dx-\int_{-1}^{1}\psi_{k}'(x)\varphi_{i}(x)w(x),\label{l325}
\end{eqnarray}
where
\begin{eqnarray}
\varphi_{i}(x)=\frac{2x\mu}{1-x^{2}}\psi_{i}(x).\label{l325b}
\end{eqnarray}
Observe that $\varphi_{i}(x_{h})=0, h=1,\ldots,N-1, h\neq i$, and that the quadrature is exact when applied to (\ref{l325}). All this leads to
\begin{eqnarray}
(\psi_{k}'',\psi_{i})_{w}&=&-\left(\sum_{h=0}^{N}\psi_{k}'(x_{h})\psi_{i}'(x_{h})w_{h}-\psi_{k}'(x_{i})\varphi_{i}(x_{i})w_{i}\right.\nonumber\\
&&\left.-(\psi_{k}'(x_{0})\varphi_{i}(x_{0})w_{0}+\psi_{k}'(x_{N})\varphi_{i}(x_{N})w_{N})\right),\label{l326}
\end{eqnarray}
with
\begin{eqnarray*}
\varphi_{i}(x_{i})=\frac{2x_{i}\mu}{1-x_{i}^{2}}, \quad i=1,\ldots,N-1,\quad 
\varphi_{i}(x_{p})=\frac{\mu}{(1+\mu)(x_{i}-x_{p})}, \quad p=0,N.
\end{eqnarray*}
Note also that in the case of Legendre polynomials it holds that $w(x)=1$ and (\ref{l326}) reduces to, \cite{Kalisch}
\begin{eqnarray*}
(\psi_{k}'',\psi_{i})_{w}=-\sum_{h=0}^{N}\psi_{k}'(x_{h})\psi_{i}'(x_{h})w_{h}.
\end{eqnarray*}
In a similar way, we integrate by parts and apply the previous arguments to have
\begin{eqnarray}
(\psi_{k}''',\psi_{i})_{w}&=&-\left(\sum_{h=0}^{N}\psi_{k}''(x_{h})\psi_{i}'(x_{h})w_{h}-\psi_{k}''(x_{i})\varphi_{i}(x_{i})w_{i}\right.\nonumber\\
&&\left.-(\psi_{k}''(x_{0})\varphi_{i}(x_{0})w_{0}+\psi_{k}''(x_{N})\varphi_{i}(x_{N})w_{N})\right),\label{l326b}
\end{eqnarray}
which reduces to only the first sum in the case of Legendre polynomials. As for the nonlinear terms, integrating by parts again and using the properties, mentioned above, of the auxiliary function (\ref{l325b}) lead to
\begin{eqnarray}
(\partial_{x}(\eta^{N}u^{N}),\psi_{i})_{w}&=&-\int_{-1}^{1}\eta^{N}(x,t)u^{N}(x,t)\left(\psi_{i}w\right)'(x)dx\nonumber\\
&=&-\left(\int_{-1}^{1}\eta^{N}(x,t)u^{N}(x,t)\psi_{i}'(x)w(x)dx\right.\nonumber\\
&&\left.+\int_{-1}^{1}\eta^{N}(x,t)u^{N}(x,t)\varphi_{i}(x)w(x)dx\right)\nonumber\\
&\approx&-\left(\sum_{h=0}^{N}\eta^{N}(x_{h},t)u^{N}(x_{h},t)\psi_{i}'(x_{h})w_{h}\right.\nonumber\\
&&\left.-\sum_{h=0}^{N}\eta^{N}(x_{h},t)u^{N}(x_{h},t)\varphi_{i}'(x_{h})w_{h}\right)\nonumber\\
&=&-\sum_{h=0}^{N}\eta_{h}(t)u_{h}(t)\psi_{i}'(x_{h})w_{h}+\left(\eta_{i}(t)u_{i}(t)\varphi(x_{i})w_{i}\right.\nonumber\\
&&\left.+\eta_{0}(t)u_{0}(t)\varphi_{i}(x_{0})w_{0}+\eta_{N}(t)u_{N}(t)\varphi_{i}(x_{N})w_{N}\right),\label{l327}
\end{eqnarray}
and similarly
\begin{eqnarray}
\frac{1}{2}(\partial_{x}(u^{N})^{2},\psi_{i})_{w}&\approx&-\sum_{h=0}^{N}\frac{u_{h}(t)^{2}}{2}\psi_{i}'(x_{h})w_{h}+\frac{u_{i}(t)^{2}}{2}\varphi(x_{i})w_{i}\nonumber\\
&&+\frac{u_{0}(t)^{2}}{2}\varphi_{i}(x_{0})w_{0}+\frac{u_{N}(t)^{2}}{2}\varphi_{i}(x_{N})w_{N}.\label{l328}
\end{eqnarray}
Note again the reduction in (\ref{l327}), (\ref{l328}) when the Legendre family of polynomials is taken. We also observe that some of the contributions of the auxiliary function (\ref{l325b}) are known data and will take part of the right-hand side of the final ode system, as we will see. Thus, (\ref{l324ab})-(\ref{l328}) are substituted into (\ref{l324a}), (\ref{l324b}) and the left-hand side of the resulting equations will involve four different matrices:
\begin{itemize}
\item The diagonal matrix $K_{N}^{(0)}={\rm diag}(w_{0},\ldots,w_{N})$ of quadrature weights.
\item The first- and second-order differentiation matrices
\begin{eqnarray*}
&&D^{(1)}=(D_{ij}^{(1)}),\quad D_{ij}^{(1)}=\psi_{j}'(x_{i}), \quad i,j=0,\ldots,N,\\
&&D^{(2)}=(D_{ij}^{(2)}),\quad D_{ij}^{(2)}=\psi_{j}''(x_{i}), \quad i,j=0,\ldots,N.
\end{eqnarray*}
\item The $(N+1)\times (N-1)$ auxiliary matrix $\Psi=(\Psi_{ij})$ with $\Psi_{ij}=\varphi_{j}(x_{i}), i=0,\ldots N, j=1,\ldots,N-1$. Thus $\Psi$ has the form
\begin{eqnarray*}
\Psi=\begin{pmatrix}\varphi_{1}(x_{0})&\varphi_{2}(x_{0})&\cdots&\varphi_{N-1}(x_{0})\\
\varphi_{1}(x_{1})&0&\cdots&0\\
0&\varphi_{2}(x_{2})&\cdots&0\\
\vdots&\vdots&\cdots&\vdots\\
0&0&\cdots&\varphi_{N-1}(x_{N-1})\\
\varphi_{1}(x_{N})&\varphi_{2}(x_{N})&\cdots&\varphi_{N-1}(x_{N})\end{pmatrix},
\end{eqnarray*}
and is not present in the case of approximating with Legendre polynomials.
\end{itemize}
Let
\begin{eqnarray*}
&&W=K_{N}^{(0)}(1:N-1,1:N-1),\quad \widetilde{D}=D^{(1)}(1:N-1,1:N-1),\\
&&\widetilde{D}^{(2)}=D^{(2)}(1:N-1,1:N-1),\\
&& \widetilde{\Psi}=\Psi(1:N-1,1:N-1)={\rm diag}(\varphi_{1}(x_{1}),\ldots,\varphi_{N-1}(x_{N-1})),
\end{eqnarray*}
where $M(i_{1}:i_{2},j_{1}:j_{2})$ denotes the submatrix of a matrix $M$ consisting of the files from $i_{1}$ to $i_{2}$ and the columns from $j_{1}$ to $j_{2}$.
Then the resulting ode system for the coefficients $\eta(t)=(\eta_{1}(t),\ldots,\eta_{N-1}(t))^{T}$ and $u(t)=(u_{1}(t),\ldots,u_{N-1}(t))^{T}$ will have the form
\begin{eqnarray}
&&\left(W+b\left(\widetilde{D}^{T}W\widetilde{D}-\widetilde{\Psi}^{T}W\widetilde{D}\right)\right)\eta'(t)+W\widetilde{D}u(t)\nonumber\\
&&-\left(\widetilde{D}^{T}W-\widetilde{\Psi}^{T}W\right)\eta(t).u(t)=\Gamma_{1},\label{l329a}\\
&&\left(W+b\left(\widetilde{D}^{T}W\widetilde{D}-\widetilde{\Psi}^{T}W\widetilde{D}\right)\right)u'(t)+\left(W\widetilde{D}+|c|\left(\widetilde{D}^{T}W\widetilde{D}^{(2)}-\widetilde{\Psi}^{T}W\widetilde{D}^{(2)}\right)\right)
\eta(t)\nonumber\\
&&
-\left(\widetilde{D}^{T}W-\widetilde{\Psi}^{T}W\right)\frac{1}{2}u(t).u(t)=\Gamma_{2},\label{l329b}
\end{eqnarray}
with
\begin{eqnarray}
\Gamma_{1}(t)&=&-b\eta_{0}'(t)\left((D^{(1)}(:,1:N-1))^{T}K_{N}^{(0)}D^{(1)}(:,0)-\Psi^{T}K_{N}^{(0)}D^{(1)}(:,0)\right)\nonumber\\
&&-b\eta_{N}'(t)\left((D^{(1)}(:,1:N-1))^{T}K_{N}^{(0)}D^{(1)}(:,N)-\Psi^{T}K_{N}^{(0)}D^{(1)}(:,N)\right)\nonumber\\
&&-u_{0}(t)WD^{(1)}(1:N-1,0)-u_{N}(t)WD^{(1)}(1:N-1,N)\nonumber\\
&&+\eta_{0}(t)u_{0}(t)w_{0}D^{(1)}(0,1:N-1)^{T}+\eta_{N}(t)u_{N}(t)w_{N}D^{(1)}(N,1:N-1)^{T}\nonumber\\
&&-\eta_{0}(t)u_{0}(t)w_{0}\Psi(0,:)^{T}-\eta_{N}(t)u_{N}(t)w_{N}\Psi(N,:)^{T},\label{l330a}\\
\Gamma_{2}(t)&=&-bu_{0}'(t)\left((D^{(1)}(:,1:N-1))^{T}K_{N}^{(0)}D^{(1)}(:,0)-\Psi^{T}K_{N}^{(0)}D^{(1)}(:,0)\right)\nonumber\\
&&-bu_{N}'(t)\left((D^{(1)}(:,1:N-1))^{T}K_{N}^{(0)}D^{(1)}(:,N)-\Psi^{T}K_{N}^{(0)}D^{(1)}(:,N)\right)\nonumber\\
&&-\eta_{0}(t)WD^{(1)}(1:N-1,0)-\eta_{N}(t)WD^{(1)}(1:N-1,N)\nonumber\\
&&-|c|\left((D^{(1)}(:,1:N-1))^{T}K_{N}^{(0)}D^{(2)}(:,0)-\Psi^{T}K_{N}^{(0)}D^{(2)}(:,0)\right)\eta_{0}(t)\nonumber\\
&&-|c|\left((D^{(1)}(:,1:N-1))^{T}K_{N}^{(0)}D^{(2)}(:,N)-\Psi^{T}K_{N}^{(0)}D^{(2)}(:,N)\right)\eta_{N}(t)\nonumber\\
&&+\frac{u_{0}(t)^{2}}{2}w_{0}D^{(1)}(0,1:N-1)^{T}+\frac{u_{N}(t)^{2}}{2}w_{N}D^{(1)}(N,1:N-1)^{T}\nonumber\\
&&-\frac{u_{0}(t)^{2}}{2}w_{0}\Psi(0,:)^{T}-\frac{u_{N}(t)^{2}}{2}w_{N}\Psi(N,:)^{T}.\label{l330b}
\end{eqnarray}
\begin{remark}
The G-NI formulation makes the Galerkin method be essentially equivalent to a collocation approach, in the sense mentioned in, e.~g. \cite{CanutoHQZ1988}, and for the spectral collocation scheme based on the Gauss-Lobatto-Jacobi nodes $x_{j}, j=0,\ldots,N$, associated to the weight $w$. For the case of (\ref{BS2}), the semidiscrete collocation approximation is defined as the mapping $(\eta^{N},u^{N}):[0,T]\rightarrow\mathbb{P}_{N}$ such that
\begin{eqnarray}
\eta_{t}^{N}-b\eta_{xxt}^{N}+(I_{N}(\eta^{N}u^{N}))_{x}+f^{N}&=&0,\nonumber\\
u_{t}^{N}-bu_{xxt}^{N}+\frac{1}{2}(I_{N}((u^{N})^{2})_{x}+g^{N}&=&0,\label{BS29}
\end{eqnarray}
at $x=x_{j}, j=1,\ldots,N-1$ with
\begin{eqnarray}
\eta^{N}(0)\Big|_{x=x_{j}}={\eta}_{0}(x_{j}),
u^{N}(0)\Big|_{x=x_{j}}={u}_{0}(x_{j}),\; j=0,\ldots,N,\label{BS210}
\end{eqnarray} 
satisfying (\ref{BS2b}), (\ref{BS2c}), and where $I_{N}v$ denotes the interpolating polynomial of $v$ on $\mathbb{P}_{N}$ based on the Gauss-Lobatto-Jacobi nodes. The implementation and analysis of (\ref{BS29}), (\ref{BS210}) make use of the discrete inner product 
\begin{eqnarray}
\left(\phi,\psi\right)_{N,w}=\sum_{j=0}^{N}\phi(x_{j})\psi(x_{j})w_{j},\label{ad12}
\end{eqnarray}
for $\phi, \psi$ continuous on $\overline{\Omega}$, with associated norm $||\phi||_{N,w}=\left(\phi,\phi\right)_{N,w}^{1/2}$, and where the weights $w_{j}$ are given in theorem \ref{theorem198}. Thus (\ref{BS29}) can be formulated from the representation (\ref{l323}) and making the inner product (\ref{ad12}) of each equation with the nodal functions $\psi_{i}, i=1,\ldots,N-1$. The equivalence with the G-NI approach is based on the identification of the approximations of the integrals with the corresponding inner products involving (\ref{ad12}).

\end{remark}

\section{Full discretization and numerical experiments}
\label{sec4}
In this section the full discretization of (\ref{BS2})-(\ref{BS2c}) is completed with a discussion on the choice of the temporal discretization and some numerical experiments to illustrate the performance of the fully discrete method. 
\subsection{Temporal discretization}
The following family of singly diagonally implicit Runge-Kutta (SDIRK) 
methods  of Butcher tableau
\begin{eqnarray}
\label{sdirk}
\begin{array}{c | cc}
\gamma& \gamma & 0  \\[2pt]
1-\gamma& 1-2\gamma& \gamma\\[2pt]
\hline
\\[-9pt]
 & \frac{1}{2}  & \frac{1}{2}
 \end{array}
\end{eqnarray}
is proposed as time integrator of the semidiscrete system (\ref{l329a})-(\ref{l330b}). For a system of ode's of the general form
\begin{eqnarray}
u^{\prime}(t)=F(u),\label{ssp1}
\end{eqnarray}
the formulation of the methods is given by
\begin{eqnarray}
u^{*}&=&u^{n}+\gamma k F(u^{*}),\nonumber\\
u^{**}&=&u^{n}+(1-2\gamma)k F(u^{*})+\gamma k F(u^{**}),\nonumber\\
u^{n+1}&=&u^{n}+\frac{k}{2}\left(F(u^{*})+F(u^{**})\right),\label{ssp2}
\end{eqnarray}
where $u^{n}$ denotes an approximation to the solution $u$ of (\ref{ssp1}) at $t=t_{n}=nk, n=0,1,\ldots$, being $k$ the time step size. The method (\ref{ssp2}) is linearly implicit, requiring the iterative resolution of two intermediate systems (although the classical fixed-point algorithm typically works). The values
$\gamma=1/2$ (implicit midpoint rule, order 
two) and $\gamma=\frac{3+\sqrt{3}}{6}$ (order three) will be considered below. Several properties justify our choice. 
First, we recall that the methods are A-stable (and therefore L-stable), cf. e.~g. \cite{hnw2}, as well as nondissipative.
On the other hand, they are dispersive of order $q=2$ and $3$ respectively. The nature of this property (called phase error analysis or phase-lag analysis in the case of second-order equations) has been studied in many references in the literature, (cf. e.~g. \cite{FrancoGR,IzzoJ2021} and references therein). A brief explanation of the theory and justification for its mention here follows. The origin is the study and design of efficient numerical integrators of ode's (\ref{ssp1}) having oscillatory solutions. Such problems typically appear in spatial discretizations of pde's of hyperbolic type or with some hyperbolic components. The goal is the analysis of the phase (or dispersion) errors and numerical dissipation of the free oscillations in the numerical solution, as well as the construction of schemes for their resolution. In order to measure these effects, the methods are studied when applied to linear systems with eigenvalues on the imaginary axis. Focused on Runge-Kutta (RK) schemes, the phase error analysis starts from the stability function $R(z)$ of the RK method which, as rational approximation to the exponential function, is called consistent of order $p$ when
$$R(z)=e^{z}+O(z^{p+1}).$$ From this stability function $R$, the dispersion or phase error is defined as, \cite{SomH},
$$\Phi(y)=y-{\rm arg}(R(iy)),\quad y\in\mathbb{R},$$ where ${\rm arg}z$ denotes an argument of $z\in\mathbb{C}$. The RK method is said to have dispersion order $q$ if
$$\Phi(y)=O(y^{q+1}),\quad y\rightarrow 0.$$
The fundamental relation between the order of consistency and the dispersion order is proved in \cite{Koto1990}
\begin{lemma}
Let $p$ be the order of consistency of a RK method with stability function $R$. Then the dispersion order is $q=p$ if $p$ is even and $q\geq p+1$ if $p$ is odd.
\end{lemma}
This result is more specific in the case of DIRK methods.
\begin{theorem}
Let $p=m+1$ be the order of consistency of a DIRK method with stability function $R$. Then the dispersion order is $q=p$ if $m$ is odd and $q= p+1$ if $m$ is even.
\end{theorem}
In particular, for the SDIRK methods of tableau (\ref{sdirk}), the dispersion order is $q=2$ when $\gamma=1/2$ and $q=3$ when $\gamma=\frac{3+\sqrt{3}}{6}$.

As mentioned in \cite{SomH}, for long time integrations of (\ref{ssp1}) with fixed stepsize, the dispersion is a relevant source of phase error causing the numerical approximation to become increasingly out of phase with the exact solution. Although the elements of the theory are defined and analyzed from linear oscillatory problems, it is expected to be of application in more general cases of (\ref{ssp1}) with solutions containing some oscillatory behaviour, in the sense that a method of high dispersion order will generate small dispersion errors of the corresponding oscillations that the numerical approximation may contain.

An additional, relevant feature of the two methods is concerned with the so-called strong stability preserving (SSP) property. This is briefly formulated as follows (see e.~g. \cite{GotliebST2001} for details). Let $u_{FE}^{n}$ be the approximation to some solution $u$ of (\ref{ssp1}) under study at $t=t_{n}=nk$, given by the forward Euler method and assume that, for some convex functional $C$ (typically some norm or seminorm) there exists $k_{FE}>0$ such that
\begin{eqnarray}
C(u_{FE}^{n+1})\leq C(u_{FE}^{n}),\label{ssp3}
\end{eqnarray}
for $k\leq k_{FE}$. From this assumption, and for a $s$-stage Runge-Kutta (RK) method
\begin{eqnarray}
y_{i}&=&u^{n}+\Delta t\sum_{j=1}^{s}a_{ij}F(y_{j}),\; 1\leq i\leq s+1,\nonumber\\
u^{n+1}&=&y_{s+1},\label{ssp4}
\end{eqnarray}
we define the
associated SSP coefficient as the largest constant $\delta\geq 0$ such that if $k\leq \delta k_{FE}$ then
\begin{eqnarray}
C(y_{i})\leq C(u^{n}),\quad 1\leq i\leq s+1.\label{ssp5}
\end{eqnarray}
In particular, this implies the monotonicity preserving property $C(u^{n+1})\leq C(u^{n})$.
If $\delta>0$, the method (\ref{ssp4}) is said to be strong stability preserving. The choice of the methods (\ref{sdirk}) with $\gamma=1/2$ and $\gamma=(3+\sqrt{3})/{6}$ in this context is justified by the observation that they are optimal SSP methods within the SDIRK schemes with the same stages and order,  in the 
sense that the value $\delta$ in property (\ref{ssp5}) for $k\leq \delta k_{FE}$
is maximal, \cite{FerracinaS2008,KetchesonMG2009}. 

\subsection{Numerical experiments}
In this section we show some numerical experiments to illustrate the accuracy and performance of the fully discrete schemes analyzed in the present paper. For the reasons explained in section \ref{sec33}, the subfamily of Legendre polynomials will be used as the spectral approximation in space. The experiments will be focused on the convergence of the methods and their performance when approximating problems with nonregular data, for which the error estimates (\ref{l315a}), (\ref{l315b})
 of Theorem \ref{theorem32} cannot be ensured.

It may be worth explaining first some details on the computations. Let $M>1$ be an integer and $f\in H^{k}(\Omega), k\geq 0$. In the computations below we will make use of the quadrature rule (\ref{l320}) based on the $M+1$ Gauss-Lobatto-Legendre nodes $x_{j}$ and weights $w_{j}, j=0,\ldots,M$, in order to estimate the $H^{k}$norms of $f$ via
\begin{eqnarray}
\int_{-1}^{1}|f^{l)}(x)|^{2}dx\approx \sum_{j=0}^{M}|f^{l)}(x_{j})|^{2}w_{j},\label{l421}
\end{eqnarray}
where $f^{l)}$ denotes the $l$th derivative of $f$. 


The specific application of (\ref{l421}) is made from the following representation of the numerical solution at some final time $t=T$. Let $N>1$ be an integer. From (\ref{l323}) the time integration with (\ref{sdirk}) provides approximations $\eta_{k}^{N}, u_{k}^{N}$ to $\eta_{k}(T), u_{k}(T)$, respectively, for $k=0,\ldots,N$. Then we consider the polynomials
\begin{eqnarray}
\widetilde{\eta}^{N}(x)=\sum_{k=0}^{N}\eta_{k}^{N}\psi_{k}(x),\quad
\widetilde{u}^{N}(x)=\sum_{k=0}^{N}u_{k}^{N}\psi_{k}(x),\label{l422}
\end{eqnarray}
as approximations to (\ref{l323}) at $t=T$. If the numerical solution (\ref{l422}) is compared with an exact solution $(\eta,u)$ at $t=T$, then the corresponding norms of the errors are computed via (\ref{l421}). For those problems of which the exact solution is not known, the order of convergence in the spatial discretization is estimated from the quotients
\begin{eqnarray}
E_{N}=\frac{||(\widetilde{\eta}^{N}-\widetilde{\eta}^{2N},\widetilde{u}^{N}-\widetilde{u}^{2N})||}{||(\widetilde{\eta}^{2N}-\widetilde{\eta}^{4N},\widetilde{u}^{2N}-\widetilde{u}^{4N})||},\label{l423}
\end{eqnarray}
where $||(\cdot,\cdot)||$ denotes some norm in the corresponding product space and which is computed using (\ref{l421}).

The first examples are focused on the convergence of the methods. We consider, for $\theta^{2}\in (7/9,1)$, the solitary wave solution of the Bona-Smith family of the form, \cite{MChen1998b} 
\begin{eqnarray}
\eta_{s}(x,t)=\eta_{0}{\rm sech}^{2}(\lambda (x-c_{s}t-x_{0})),\quad u_{s}(x,t)=B\eta_{s}(x,t),\label{421}
\end{eqnarray}
with
\begin{eqnarray*}
&&\eta_{0}=\frac{9}{2}\frac{\theta^{2}-7/9}{1-\theta^{2}},\quad c_{s}=\frac{4(\theta^{2}-2/3)}{\sqrt{2(1-\theta^{2})(\theta^{2}-1/3)}},\\
&&\lambda=\frac{1}{2}\sqrt{\frac{3(\theta^{2}-7/9)}{(\theta^{2}-1/3)(\theta^{2}-2/2)}},\quad B=\sqrt{\frac{2(1-\theta^{2})}{\theta^{2}-1/3}}.
\end{eqnarray*}
The solutions decay to zero exponentially as $x-c_{s}t\rightarrow\pm\infty$; therefore, they will be taken as solutions of an ibvp with homogeneous Dirichlet boundary conditions on a long enough interval $[-L,L]$.
\begin{table}[htbp]
\begin{center}
\begin{tabular}{|c|c|c|c|c|}
    \hline
&    \multicolumn{2}{c|} {$\gamma=1/2$}& \multicolumn{2}{c|}{$\gamma=\frac{3+\sqrt{3}}{6}$}\\
\hline
{$k$} &{$H^{2}\times H^{1}$ Error}&{Rate}&{$H^{2}\times H^{1}$ Error}&{Rate}\\
\hline
1.25E-01&2.2373E-02&&6.7487E-03&\\
6.25E-02&5.6737E-03&1.98&8.7667E-04&2.96\\
3.125E-02&1.4236E-03&1.99&1.1029E-04&2.99\\
    \hline
\end{tabular}
\end{center}
\caption{Numerical approximation of (\ref{421}) with $\theta^{2}=9/11$ on $[-32,32]$ with homogeneous Dirichlet boundary conditions: $H^{2}\times H^{1}$ 
norms of the error at $T=2$ and rates of convergence with Legendre Galerkin 
method where $N=512$.}
\label{adtav1}
\end{table}

Table \ref{adtav1} shows the error with respect to (\ref{421}) for $\theta^{2}=9/11$, measured in the $H^{2}\times H^{1}$ norm, and performed by the two time integrators, for an interval $[-L,L], L=32$ at $T=2$ and using $N=512$ Gauss-Lobatto-Legendre nodes in the G-NI formulation. Note that, as a consequence of the regularity of (\ref{421}) and the error estimate (\ref{l315a}), spectral convergence in space is expected. This is suggested by the rates shown in third and fifth columns, which correspond to the order of convergence of each time integrator, as well as the property checked that larger values of $N$ do not change this behaviour.

\begin{table}[htbp]
\begin{center}
\begin{tabular}{|c|c|c|c|c|}
    \hline
&    \multicolumn{2}{c|} {$\gamma=1/2$}& \multicolumn{2}{c|}{$\gamma=\frac{3+\sqrt{3}}{6}$}\\
\hline
{$k$} &{$H^{2}\times H^{2}$ Error}&{Rate}&{$H^{2}\times H^{2}$ Error}&{Rate}\\
\hline
1.25E-01&2.6200E-02&&6.8554E-03&\\
6.25E-02&6.6298E-03&1.98&8.6027E-04&2.99\\
3.125E-02&1.6627E-03&1.99&1.0989E-04&2.98\\
    \hline
\end{tabular}
\end{center}
\caption{Numerical approximation of (\ref{422}) with $\rho=2, c_{s}=1$ on $[-16,16]$ with homogeneous Dirichlet boundary conditions: $H^{2}\times H^{2}$ 
norms of the error at $T=2$ and rates of convergence with Legendre Galerkin 
method where $N=256$.}
\label{adtav2}
\end{table}

The spectral convergence in space is also observed in Table \ref{adtav2}. The comparison here is established with respect to the traveling wave solutions of the BBM-BBM system ($c=0, b=d=1/6$), \cite{MChen1998a}
\begin{eqnarray}
\eta_{s}(x,t)&=&-1+(c_{s}b\rho)^{2}(\frac{4}{9}+\frac{5}{3}w(x,t)(2-3w(x,t))),\nonumber\\
u_{s}(x,t)&=&\frac{c_{s}}{3}(3-5b\rho)+5c_{s}b\rho w(x,t),\label{422}
\end{eqnarray}
with
\begin{eqnarray*}
w(x,t)={\rm sech}^{2}\left(\frac{\sqrt{\rho}}{2}(x-c_{s}t-x_{0}\right),
\end{eqnarray*}
where $\rho>0, c_{s}\neq 0, x_{0}\in\mathbb{R}$. In this case, the errors are measured in the $H^{2}\times H^{2}$ norm and the results are concerned with the estimate (\ref{l315b}).

 \begin{table}[htbp]
\begin{center}
\begin{tabular}{|c|c|c|c|c|}
    \hline
&    \multicolumn{2}{c|} {$\gamma=1/2$}& \multicolumn{2}{c|}{$\gamma=\frac{3+\sqrt{3}}{6}$}\\
\hline
{$k$} &{$H^{2}\times H^{2}$ Error}&{Rate}&{$H^{2}\times H^{2}$ Error}&{Rate}\\
\hline
1.25E-01&3.6446E-02&&1.0898E-02&\\
6.25E-02&9.2402E-03&1.98&1.4150E-03&2.95\\
3.125E-02&2.3183E-03&1.99&1.7802E-04&2.99\\
    \hline
\end{tabular}
\end{center}
\caption{Numerical approximation of (\ref{423}) with $\eta_{0}=1$ on $[-32,32]$ with homogeneous Dirichlet boundary conditions: $H^{2}\times H^{2}$ 
norms of the error at $T=2$ and rates of convergence with Legendre Galerkin 
method where $N=512$.}
\label{adtav3}
\end{table}

We observe, on the other hand, that
the proof of Theorem \ref{theorem32} suggests that the error estimates (\ref{l315a}), (\ref{l315b}) also hold for other cases of the Boussinesq systems (\ref{BS1}), (\ref{BS1p}) out of the Bona-Smith family; in particular, those with $b\neq d$. (This can also be applied to the theoretical result of well-posedness given by Theorem \ref{theorem25}, cf. \cite{ADM2}.) In order to check this point, we will consider an example of (\ref{BS1}), (\ref{BS1p}) with $\mu=\nu=0$, so that
\begin{eqnarray*}
a=c=0,\quad b=\frac{1}{2}(\theta^{2}-1/3),\quad d=\frac{1}{2}(1-\theta^{2}),
\end{eqnarray*}
and $\theta^{2}=7/9$. The resulting system admits solitary wave solutions of the form, \cite{MChen1998a}
\begin{eqnarray}
\eta_{s}(x,t)=\eta_{0}{\rm sech}^{2}(\lambda (x-c_{s}t-x_{0})),\quad u_{s}(x,t)=u_{0}{\rm sech}^{2}(\lambda (x-c_{s}t-x_{0})),\label{423}
\end{eqnarray}
with $\eta_{0}>-3$ such that $3/(\eta_{0}+3)\notin [1,2]$ and
\begin{eqnarray*}
u_{0}=\eta_{0}\sqrt{\frac{3}{3+\eta_{0}}},\quad c_{s}=\frac{3+2\eta_{0}}{\sqrt{3(3+\eta_{0})}},\quad
\lambda=\frac{1}{2}\sqrt{\frac{2\eta_{0}}{b(3+2\eta_{0})}}.
\end{eqnarray*}
The $H^{2}\times H^{2}$ norm of the error when approximating (\ref{423}) for $\eta_{0}=1$ with the fully discrete methods is shown in Table \ref{adtav3}. The results suggest that the spectral convergence in space of the approximations also holds for this case, since larger values of the number $N$ of discretization points did not changed the rates of convergence.

We complete this first group of experiments with the numerical approximation of the BBM-BBM system on some interval $[a,b], a<0, b>0$, boundary conditions
\begin{eqnarray}
\eta(a,t)=\eta_{0},\quad \eta(b,t)=0,\quad u(a,t)=u_{0},\quad u(b,t)=0,\quad t>0,\label{l424}
\end{eqnarray}
where  $\eta_{0}>0$ and $u_{0}=\frac{\eta_{0}}{\eta_{0}+1}\sqrt{\frac{1}{2}(2+3\eta_{0}+\eta_{0}^{2}}$, and initial conditions
\begin{eqnarray}
\eta_{0}(x,0)=\frac{\eta_{0}}{2}(1-{\rm tanh}(\kappa x)),\quad
u(x,0)=\frac{u_{0}}{2}(1-{\rm tanh}(\kappa x)),\label{l425}
\end{eqnarray}
for some parameter $\kappa$. The initial data evolves to the generation of an undular bore of some type, \cite{Kalisch,Favre,Kalisch2,Peregrine}. The approximate undular bore is illustrated in Figure \ref{figL1}. The time integrator is the third-order scheme (\ref{sdirk}) with time stepsize $k=0.1h, h=6.25E-03$.

\begin{figure}[htbp]
\centering
\subfigure[]
{\includegraphics[width=0.47\textwidth]{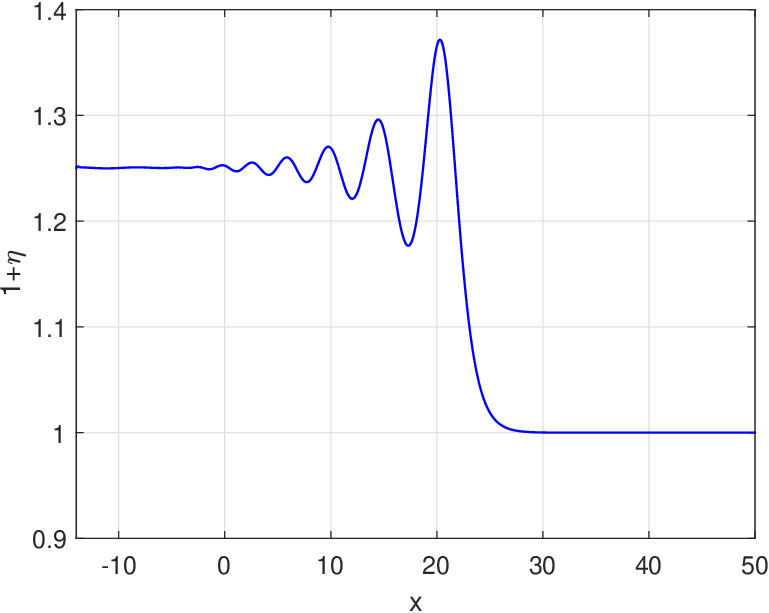}}
\subfigure[]
{\includegraphics[width=0.47\textwidth]{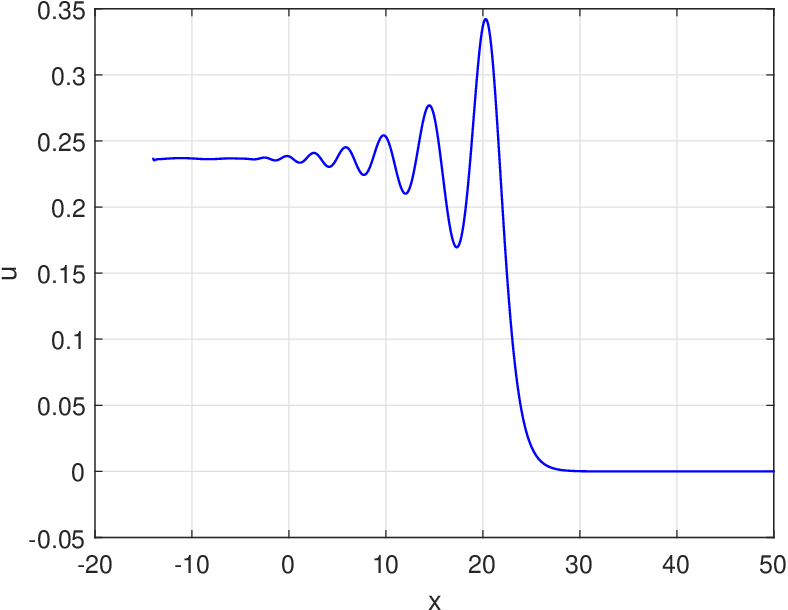}}
\caption{Numerical solution at $t=20$ of the BBM-BBM system on $[-14,50]$, from (\ref{l425}), (\ref{l424}), with $\eta_{0}=0.25,\kappa=0.7$. (a) $1+\eta$ component; (b) $u$ component.}
\label{figL1}
\end{figure}

\begin{table}[htbp]
\begin{center}
\begin{tabular}{|c|c|c|}
\hline
{$N$} &{$E_{N}$ ($L^{2}\times L^{2}$ Error)}&{$\log(E_{N})$}\\
\hline
64&6.1602E-01&-4.8448E-01\\
128&6.1445E-01&-4.8703E-01\\
256&6.1440E-01&-4.8712E-01\\
    \hline
\end{tabular}
\end{center}
\caption{Numerical solution at $t=20$ of the BBM-BBM system on $[-14,50]$, from (\ref{l425}), (\ref{l424}), with $\eta_{0}=0.25,\kappa=0.7$. Ratios (\ref{l423}) and corresponding logarithms from $N=64$.}
\label{adtav4}
\end{table}
Since the initial data (\ref{l425}) is smooth, an exponential order of convergence in space is expected. Table \ref{adtav4} shows, for $T=20$, the quotients (\ref{l423}) (using the $L^{2}\times L^{2}$ norm) and their logarithms from $N=64$ obtained with the third-order scheme (\ref{sdirk}). The results suggest that the errors decay as $e^{-h/2}$, where $h$ is the stepsize in space.

A second group of experiments illustrates the performance of the schemes with nonregular data. Recall that the error estimates  (\ref{l315a}), (\ref{l315b}) depend on the regularity of the solution of (\ref{BS2}) to be approximated. In this sense, for the nonsmooth case, some reduction of order of convergence and the generation of spurious oscillations in the numerical approximation are expected. All the experiments below are performed to approximate (\ref{BS2}) in $\Omega=[-1,1]$ with homogeneous boundary conditions, usng the third-order scheme (\ref{sdirk}) for $\gamma=(3+\sqrt{3})/6$ and $k=0.1h$.

We first study the cases $\theta^{2}=2/3$ (BBM-BBM system) and $\theta^{2}=9/11$, with initial conditions
\begin{eqnarray}
\eta_{0}(x)=\left\{\begin{matrix}1+2x+x^{2}&-1\leq x\leq 0\\1+2x-3x^{2}&0\leq x\leq 1\end{matrix}\right.,
\quad
u(x,0)=\eta_{0}(x).\label{l426}
\end{eqnarray}
Note that the second derivative of $\eta_{0}$ has a discontinuity at $x=0$ and the presence of the term $c\eta_{xxx}$ should affect the regularity of the solution. The corresponding approximate solutions at $T=1$ are shown in Figures \ref{figL2} and \ref{figL3}, respectively.

\begin{figure}[htbp]
\centering
\subfigure[]
{\includegraphics[width=0.47\textwidth]{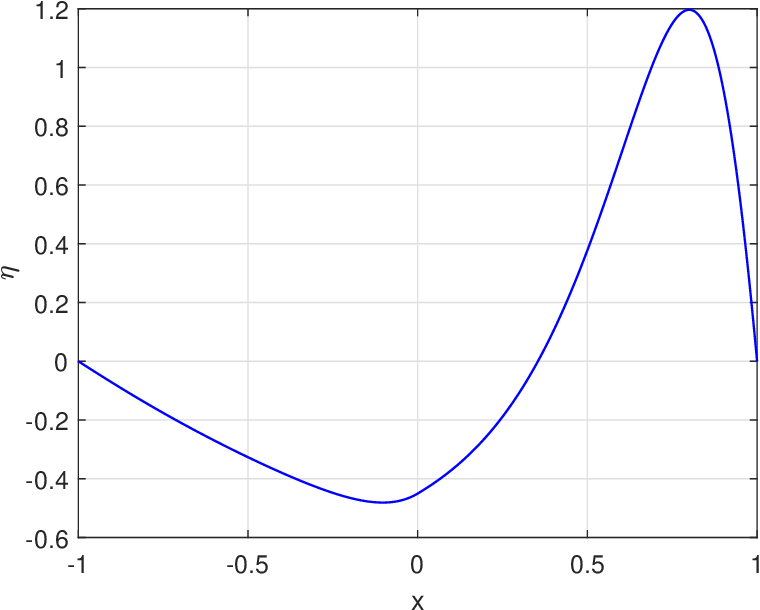}}$\;\;$
\subfigure[]
{\includegraphics[width=0.47\textwidth]{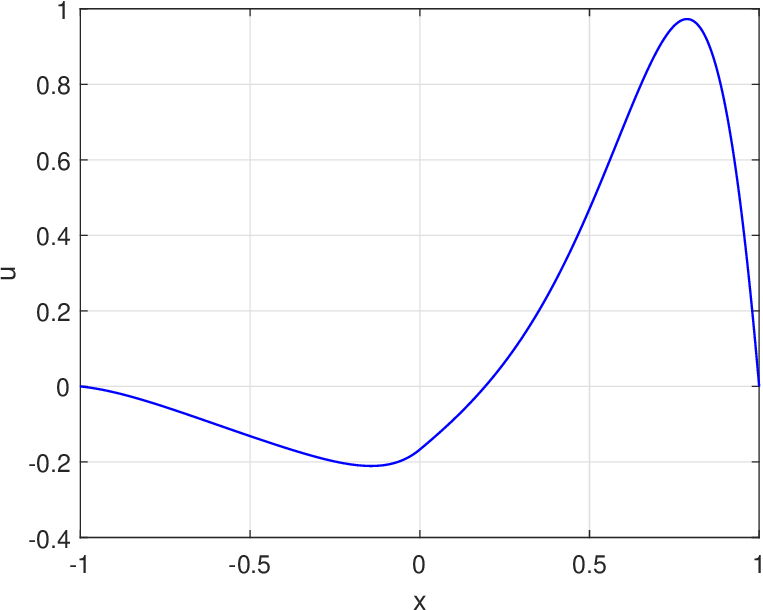}}
\caption{Numerical solution at $t=1$ of the system (\ref{BS2}) with $\theta=2/3$, zero boundary conditions and initial data (\ref{l426}) with $N=1024$.}
\label{figL2}
\end{figure}

\begin{figure}[htbp]
\centering
\subfigure[]
{\includegraphics[width=0.47\textwidth]{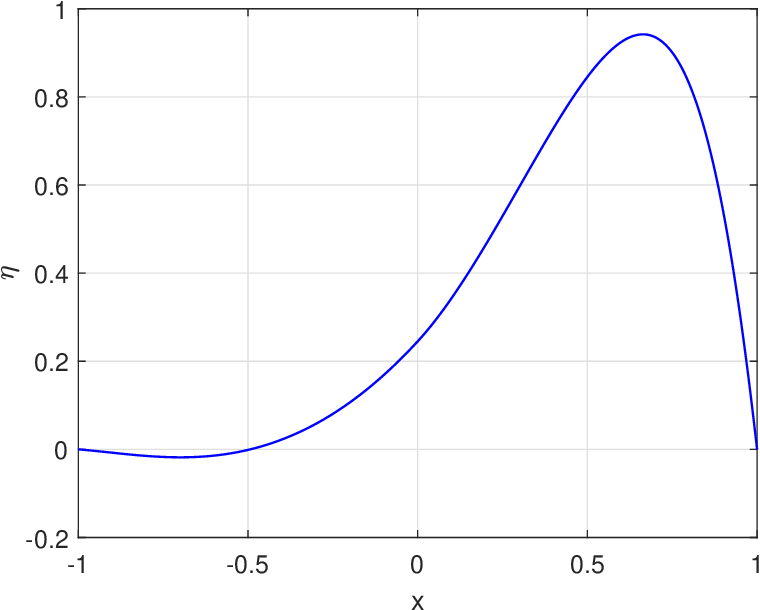}}$\;\;$
\subfigure[]
{\includegraphics[width=0.47\textwidth]{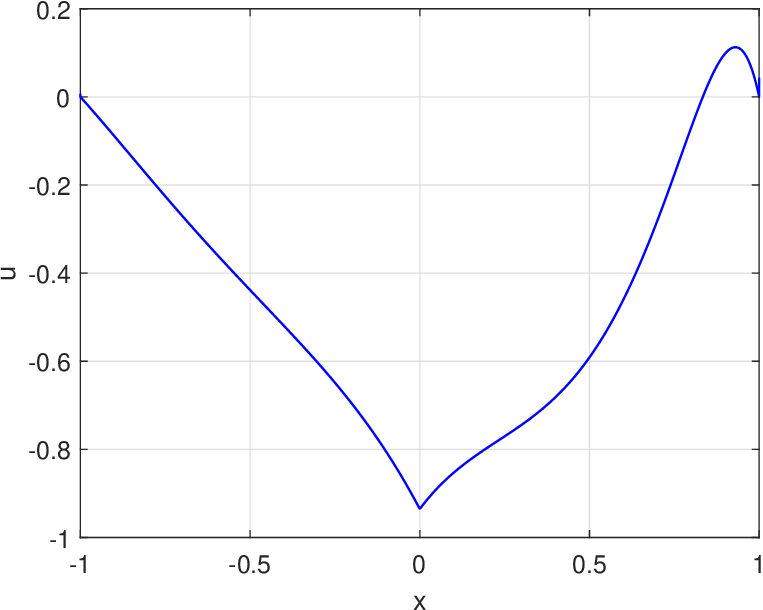}}
\caption{Numerical solution at $T=1$ of the system (\ref{BS2}) with $\theta=9/11$, zero boundary conditions and initial data (\ref{l426}) with $N=1024$.}
\label{figL3}
\end{figure}

\begin{table}[htbp]
\begin{center}
\begin{tabular}{|c|c|c|c|c|}
    \hline
&    \multicolumn{2}{c|} {$\theta^{2}=2/3$ ($H^{1}\times H^{1}$ norm)}& \multicolumn{2}{c|}{$\theta^{2}=9/11$ ($H^{1}\times L^{2}$ norm)}\\
\hline
{$N$} &{$E_{N}$}&{$\log(E_{N})/\log(2)$}&{$E_{N}$}&{$\log(E_{N})/\log(2)$}\\
\hline
16&2.6364&1.3985&2.8248&1.4981\\
32&2.7760&1.4730&2.8283&1.4999\\
64&2.8077&1.4894&2.8222&1.4968\\
128&2.8178&1.4946&2.8227&1.4971\\
    \hline
\end{tabular}
\end{center}
\caption{Spatial order of convergence at $T=1$ for the approximation of the system (\ref{BS2}) with $\theta^{2}=2/3,9/11$, zero boundary conditions and initial data (\ref{l426}).}
\label{adtav5}
\end{table}

We compute (\ref{l423}) at final time $T=1$, using the $H^{1}\times H^{1}$ norm when $\theta^{2}=2/3$ and the $H^{1}\times L^{2}$ norm  when $\theta^{2}=9/11$. The results, shown in Table \ref{adtav5}, suggest in both cases some reduction of the order of spatial convergence (recall that the example is out of the hypotheses of Theorem \ref{theorem32}) and a spatial error of $O(N^{-3/2})$.

In the last experiments we approximate the BBM-BBM system with initial conditions
\begin{eqnarray}
\eta_{0}(x)=1-|x|,\quad u_{0}(x)=0,\quad x\in\Omega.\label{l427}
\end{eqnarray}
The numerical solution at $T=1$ is shown in Figure \ref{figL4}. Note that now the derivative of $\eta-{0}$ has a discontinuity at $x=0$, but the evolution seems to regularize the $\eta-$profile, at least to allow calculating the $H^{1}$norm of the approximation.

\begin{figure}[htbp]
\centering
\subfigure[]
{\includegraphics[width=0.47\textwidth]{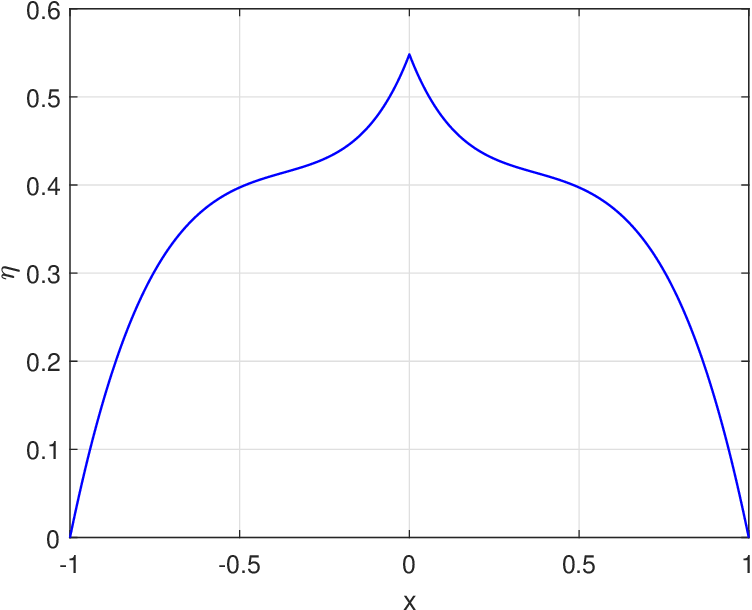}}$\;\;$
\subfigure[]
{\includegraphics[width=0.47\textwidth]{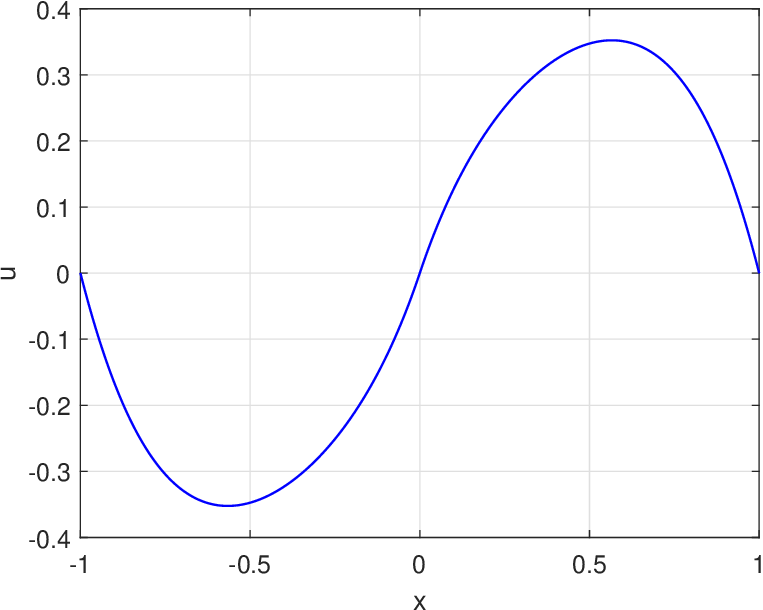}}
\caption{Numerical solution at $T=1$ of the system (\ref{BS2}) with $\theta=2/3$, zero boundary conditions and initial data (\ref{l427}) with $N=1024$.}
\label{figL4}
\end{figure}
Therefore, the values of (\ref{l423}) can be computed in the $L^{2}\times L^{2}$ and $H^{1}\times H^{1}$ norms. The results are shown in Table \ref{adtav6}, and suggest an error in space of $O(N^{-3/2})$ in the first case and of $O(N^{-1/2})$ in the second.

\begin{table}[htbp]
\begin{center}
\begin{tabular}{|c|c|c|c|c|}
    \hline
&    \multicolumn{2}{c|} {$L^{2}\times L^{2}$ norm}& \multicolumn{2}{c|}{$H^{1}\times H^{1}$ norm}\\
\hline
{$N$} &{$E_{N}$}&{$\log(E_{N})/\log(2)$}&{$E_{N}$}&{$\log(E_{N})/\log(2)$}\\
\hline
16&2.6095&1.3838&1.3861&4.7098E-01\\
32&2.7059&1.4361&1.4006&4.8606E-01\\
64&2.7656&1.4676&1.4087&4.9432E-01\\
128&2.7969&1.4838&1.4119&4.9767E-01\\
256&2.8128&1.4920&1.4132&4.9897E-01\\
    \hline
\end{tabular}
\end{center}
\caption{Spatial order of convergence at $T=1$ for the approximation of the system (\ref{BS2}) with $\theta^{2}=2/3$, zero boundary conditions and initial data (\ref{l427}).}
\label{adtav6}
\end{table}
\section*{Acknowledgements}
The author is supported by the 
Spanish Agencia Estatal de Investigaci\'on under Research Grant 
PID2020-113554GB-I00/AEI/10.13039/501100011033 and by the Junta de Castilla y Le\'on and FEDER funds (EU) 
under Research Grant VA193P20.

\end{document}